\documentclass[11pt,a4paper]{article}

\usepackage{a4wide}
\usepackage{amsmath}
\usepackage{amssymb}
\usepackage{booktabs}       
\usepackage{algorithm}
\usepackage{algorithmic}

\usepackage[utf8]{inputenc} 
\usepackage[T1]{fontenc}    
\usepackage{hyperref}       

\usepackage{amsfonts}       
\usepackage{graphicx}       
\usepackage{cleveref}
\usepackage{amsthm}
\usepackage{url}            
\usepackage{cite}
\usepackage{mathrsfs}

\usepackage{nicefrac}       
\usepackage{microtype}      
\usepackage{fancyhdr}       

\DeclareMathOperator{\diag}{diag}
\DeclareMathOperator{\Exp}{Exp}
\DeclareMathOperator{\grad}{grad}
\DeclareMathOperator{\Grass}{Grass}

\DeclareMathOperator{\qf}{qf}
\DeclareMathOperator{\Rank}{rank}

\let\skew\relax
\DeclareMathOperator{\skew}{skew}

\DeclareMathOperator{\SPD}{SPD}
\DeclareMathOperator{\St}{St}

\DeclareMathOperator{\Sym}{Sym}
\DeclareMathOperator{\tr}{tr}

\def\D{\mathrm{D}}
\def\id{\mathrm{id}}
\def\P{\mathrm{P}}

\def\FR{{\textrm{FR}}}

\def\PRP{{\textrm{PRP}}}
\def\DY{{\textrm{DY}}}
\def\HS{{\textrm{HS}}}
\def\HZ{{\textrm{HZ}}}
\def\CD{{\textrm{CD}}}
\def\LS{{\textrm{LS}}}
\def\bw{{\textrm{bw}}}
\def\RFR{{\textrm{R-FR}}}
\def\RPRP{{\textrm{R-PRP}}}
\def\RDY{{\textrm{R-DY}}}
\def\RHS{{\textrm{R-HS}}}

\def\RCD{{\textrm{R-CD}}}
\def\RLS{{\textrm{R-LS}}}
\def\RPRPFR{{\textrm{R-PRP--FR}}}
\def\RHSDY{{\textrm{R-HS--DY}}}
\def\RLSCD{{\textrm{R-LS--CD}}}
\def\RSD{{\textrm{R-SD}}}
\def\M{\mathcal{M}}
\def\N{\mathcal{N}}
\def\T{\mathscr{T}}
\def\S{\mathscr{S}}

\newtheorem{proposition}{Proposition}[section]
\newtheorem{theorem}{Theorem}[section]

\newtheorem{definition}{Definition}[section]

\theoremstyle{remark}
\newtheorem{assumption}{Assumption}[section]
\newtheorem{example}{Example}[section]
\newtheorem{problem}{Problem}[section]
\newtheorem{remark}{Remark}[section]

\crefname{proposition}{Proposition}{Propositions}
\crefname{theorem}{Theorem}{Theorems}
\crefname{corollary}{Corollary}{Corollaries}
\crefname{lemma}{Lemma}{Lemmas}
\crefname{definition}{Definition}{Definitions}
\crefname{remark}{Remark}{Remarks}
\crefname{assumption}{Assumption}{Assumptions}
\crefname{example}{Example}{Examples}
\crefname{problem}{Problem}{Problems}
\crefname{remark}{Remark}{Remarks}
\crefname{section}{Section}{Sections}
\crefname{subsection}{Section}{Sections}
\crefname{subsubsection}{Section}{Sections}
\crefname{figure}{Figure}{Figures}

\title{Riemannian conjugate gradient methods:\\ General framework and specific algorithms\\with convergence analyses\footnotetext{\textbf{Funding:} This work was funded by JSPS KAKENHI Grant number JP20K14359.}}

\author{Hiroyuki Sato\thanks{Department of Applied Mathematics and Physics, Kyoto University, Kyoto, Japan\newline ({\tt hsato@i.kyoto-u.ac.jp}).}}

\begin{document}
\maketitle

\begin{abstract}
This paper proposes a novel general framework of Riemannian conjugate gradient methods, that is, conjugate gradient methods on Riemannian manifolds.
The conjugate gradient methods are important first-order optimization algorithms both in Euclidean spaces and on Riemannian manifolds.
While various types of conjugate gradient methods are studied in Euclidean spaces, there have been fewer studies on those on Riemannian manifolds.
In each iteration of the Riemannian conjugate gradient methods, the previous search direction must be transported to the current tangent space so that it can be added to the negative gradient of the objective function at the current point.
There are several approaches to transport a tangent vector to another tangent space.
Therefore, there are more variants of the Riemannian conjugate gradient methods than the Euclidean case.
In order to investigate them in more detail, the proposed framework unifies the existing Riemannian conjugate gradient methods such as ones utilizing a vector transport or inverse retraction and also develops other methods that have not been covered in previous studies.
Furthermore, sufficient conditions for the convergence of a class of algorithms in the proposed framework are clarified.
Moreover, the global convergence properties of several specific types of algorithms are extensively analyzed.
The analyses provide the theoretical results for some algorithms in a more general setting than the existing studies and completely new developments for the other algorithms.
Numerical experiments are performed to confirm the validity of the theoretical results.
The results also compare the performances of several specific algorithms in the proposed framework.
\end{abstract}


\section{Introduction}
\label{sec:intro}
Riemannian optimization (i.e., optimization on Riemannian manifolds) has recently attracted increasing attention owing to its vast variety of applications, including machine learning, control engineering, and numerical linear algebra~\cite{AbsMahSep2008,boumal2020introduction,sato2021riemannian}.
While constrained Riemannian optimization problems have also been studied, we focus on unconstrained Riemannian optimization problems in this paper.
The class of unconstrained Riemannian optimization problems is also important since it overlaps the class of constrained optimization problems in Euclidean space.
This is because a constrained Euclidean optimization problem can be regarded as an unconstrained Riemannian optimization problem if the set of constraints forms a Riemannian manifold.
An important example is the Stiefel manifold $\St(p, n) := \{X \in \mathbb{R}^{n \times p} \mid X^T X = I_p\}$, where $p \le n$.
Any optimization problem in $\mathbb{R}^{n \times p}$ with the constraint $X^T X = I_p$ (and without any other constraint) on the decision variable matrix $X$ can be considered as an unconstrained optimization problem on $\St(p,n)$.
Another example is the manifold $\SPD(n)$, which comprises all $n \times n$ symmetric positive definite matrices.
Furthermore, the class of unconstrained Riemannian optimization problems also covers problems that are not defined in Euclidean space, e.g., optimization problems on the Grassmann manifold $\Grass(p,n) := \{W \subset \mathbb{R}^{n} \mid \text{$W$ is a $p$-dimensional subspace of $\mathbb{R}^n$}\}$ for $p \le n$, whose decision variable $W$ is a subspace of $\mathbb{R}^n$, not a vector or matrix.

An important feature of unconstrained Riemannian optimization problems is that they may be solved using the generalized versions of Euclidean unconstrained optimization methods, which have been studied intensively, if we successfully generalize the Euclidean methods to those on Riemannian manifolds appropriately.
Various Euclidean optimization methods have been generalized to the Riemannian case~\cite{absil2019collection,AbsMahSep2008,boumal2020introduction,edelman1998geometry,sato2021riemannian}.
Among them, first-order methods (such as the steepest descent (SD) and conjugate gradient (CG) methods) are important because the computational cost for each iteration is comparatively low.
In contrast, second-order methods (such as Newton's method) can have a property that they converge relatively quickly.
However, their computational cost for each iteration can be high.
For example, in each iteration of Newton's method on an $n$-dimensional Riemannian manifold, Newton's equation, which is a linear system of $n$ dimension, should be solved.
This costs at least $O(n^3)$ flops.
Furthermore, Newton's method does not guarantee the global convergence of the generated sequence.
This indicates that we must find an approximate solution that is sufficiently close to an optimal solution in advance.
Then, Newton's method with the approximate solution can generate a sequence converging to the solution.
Therefore, first-order methods retain their importance in that they can find an approximate solution in a moderately fast time.

In addition, first-order methods are appealing when large-scale optimization problems are to be solved because each iteration computationally costs much less than the second-order methods.
Among first-order methods, the framework of CG methods is one of the most important ones.
Regarding the Euclidean spaces, CG methods have been being studied and developed.
They can be regarded as the modified version of the SD method, which is a simple and basic algorithm.
Therefore, CG methods are not so complicated either; nevertheless, CG methods outperform the SD method.

In this paper, we address the CG methods on Riemannian manifolds, which we refer to as the Riemannian conjugate gradient (R-CG) methods.
Several types of R-CG methods have been studied~\cite{edelman1998geometry,edelman1996conjugate,lichnewsky1979une,ring2012optimization,sakai2020hybrid,sakai2021sufficient,sato2016dai,sato2021riemannian,sato2015new,smith1994optimization,zhu2017riemannian,zhu2020riemannian,zhu2021cayley}.
In some of these studies, parallel translation along the geodesics are utilized.
This type of approach is theoretically natural; nonetheless, there is room for computational improvement.
Other studies use a more general map called vector transport.
Using a vector transport typically enables the execution of each iteration of R-CG methods more easily than using parallel translation, whereas it may negatively affect or sometimes destroy the convergence property of the algorithm.
Therefore, in this paper, we provide a general framework of the R-CG methods, which includes all the successful existing methods, and we clarify the conditions with which the R-CG methods have a good convergence property.

The contributions of this paper are two-fold:
(i) We provide a novel general framework of the R-CG methods, which unifies all the existing R-CG methods and covers a wider class, and we clarify the assumptions that are naturally required to apply an R-CG method to a Riemannian optimization problem.
(ii) We generalize several types of standard Euclidean CG methods to the Riemannian case in our proposed framework. We also provide global convergence analyses for some specific practical algorithms.
Although we do not generalize all the Euclidean CG methods in this paper because there are various algorithms, this paper definitely provides a basis for studies on R-CG methods.

This paper is organized as follows.
In the remainder of this section, we introduce the notation used throughout the paper.
In \cref{sec:2}, we review the Euclidean CG and some existing R-CG methods, and we clarify what should be further addressed for the existing methods using some examples as a motivation for this study.
In \cref{sec:3}, we propose our new general framework of R-CG methods.
Thereafter, we introduce several types of practical R-CG methods as examples of the proposed framework.
Some conditions that are imposed on step lengths are also proposed.
In \cref{sec:4}, we summarize some standard assumptions for a Riemannian optimization problem to be solved.
We also generalize Zoutendijk's theorem to our framework.
In \cref{sec:5}, we provide the convergence analyses of several types of R-CG methods and discuss their behavior in our framework.
\cref{sec:6} provides the results of some numerical experiments of different types of R-CG methods, in which they are compared.
Finally, we conclude the paper in \cref{sec:7}.

\subsection{Notation}
The tangent space of a manifold $\M$ at $x \in \M$ is denoted as $T_x \M$, and the tangent bundle of $\M$ is denoted as $T\M := \{(x, \eta) \mid \eta \in T_x \M, \ x \in \M\}$.
For a map $F \colon \M \to \N$ between two manifolds $\M$ and $\N$, $\D F(x) \colon T_x \M \to T_{F(x)}\N$ denotes the derivative of $F$ at $x \in \M$.

In what follows, $\M$ denotes a Riemannian manifold with a Riemannian metric $\langle \cdot, \cdot\rangle$; therefore, the tangent space $T_x \M$ at any $x \in \M$ is an inner product space with the inner product $\langle \cdot, \cdot\rangle_x$, which is given by the Riemannian metric $\langle \cdot, \cdot\rangle$.
In the tangent space $T_x \M$ with the inner product $\langle \cdot, \cdot\rangle_x$, the norm of $\eta \in T_x \M$ is defined as $\|\eta\|_x := \sqrt{\langle \eta, \eta\rangle_x}$.
The Riemannian gradient $\grad f(x)$ of a function $f \colon \M \to \mathbb{R}$ at $x \in \M$ is defined as a unique tangent vector at $x$ satisfying $\langle \grad f(x), \eta\rangle_x = \D f(x)[\eta]$ for any $\eta \in T_x \M$.

The Euclidean space $\mathbb{R}^n$ with the standard inner product can be regarded as a Riemannian manifold with the Riemannian metric $\langle \cdot, \cdot\rangle$ defined by $\langle \xi, \eta\rangle_x := \xi^T \eta$ for any $\xi, \eta \in T_x \mathbb{R}^n \simeq \mathbb{R}^n$ and $x \in \mathbb{R}^n$.
In the Euclidean space $\mathbb{R}^n$ with this Riemannian metric, the Riemannian gradient $\grad f(x)$ of $f \colon \mathbb{R}^n \to \mathbb{R}$ at $x \in \mathbb{R}^n$ is equal to $\nabla f(x) := (\partial f(x) / \partial x_i) \in \mathbb{R}^n$.
We refer to this as the Euclidean gradient\footnote{In Riemannian geometry, the symbol $\nabla$ usually denotes an affine connection on Riemannian manifold $\M$. However, because we do not explicitly use affine connections in this paper and we sometimes compare the Euclidean and Riemannian CG methods, we distinguish between the Euclidean and Riemannian gradients of function $f$ in $\mathbb{R}^n$ and on $\M$ by writing them as $\nabla f$ and $\grad f$, respectively.}.
The Euclidean norm (i.e., the $2$-norm) of $a \in \mathbb{R}^n$ is denoted by $\|a\|_2 := \sqrt{a^T a}$.

We consider the following unconstrained optimization problem for minimizing a sufficiently smooth\footnote{In \cref{sec:4}, we clarify the condition on smoothness of $f$ required in convergence analyses in \cref{sec:5}.} objective function $f \colon \M \to \mathbb{R}$ defined on a Riemannian manifold~$\M$:
\begin{problem}
\label{prob:gen}
\begin{alignat*}{2}
&\text{minimize}&\quad &f(x)\\
&\text{subject to}&\quad &x \in \M.
\end{alignat*}
\end{problem}
In most Riemannian optimization algorithms, we use a retraction~$R \colon T\M \to \M$, which is a generalization of the exponential map on $\M$~\cite{AbsMahSep2008,adler2002newton}.
With the notation $R_x := R|_{T_x \M} \colon T_x \M \to \M$, a retraction $R$ on $\M$ is defined to satisfy $R_x(0_x) = x$ and $\D R_x(0_x) = \id_{T_x \M}$ for all $x \in \M$, where $0_x$ and $\id_{T_x \M}$ are the zero vector of $T_x \M$ and identity map in $T_x \M$, respectively.

\section{Euclidean CG and existing Riemannian CG methods}
\label{sec:2}
In this section, we review the linear and nonlinear CG methods in Euclidean spaces and their extension to Riemannian optimization, i.e., R-CG methods.
In particular, we discuss what have been done so far and what should be further studied in generalizing the Euclidean CG methods to Riemannian ones.

\subsection{Euclidean CG methods}
The CG method was originally proposed as an algorithm to solve linear equations of the form $Ax = b$ with an $n \times n$ symmetric positive definite matrix $A$ and $b \in \mathbb{R}^n$~\cite{hestenes1952methods}.
This method, sometimes known as the linear CG method, minimizes the strictly convex quadratic function $f(x) := x^T A x / 2 - b^T x$ for $x \in \mathbb{R}^n$ to find the minimum point $x_*$ of $f$, where $x_* = A^{-1}b$ because $\nabla f(x) = Ax - b$.
Therefore, we consider the linear CG method as an optimization algorithm for specific type of problems in Euclidean spaces.

Thereafter, the linear CG method was extended to algorithms to solve the optimization problems with more general objective functions~\cite{fletcher1964function}.
We refer to such extended algorithms as nonlinear CG methods or simply CG methods.
To solve \cref{prob:gen} with $\M = \mathbb{R}^n$ and a general objective function $f \colon \mathbb{R}^n \to \mathbb{R}$, (nonlinear) CG methods with the initial point $x_0 \in \mathbb{R}^n$ generate a sequence $\{x_k\}$ in $\mathbb{R}^n$ by
\begin{equation}
\label{eq:E_update}
x_{k+1} = x_k + t_k\eta_k, \qquad k \ge 0.
\end{equation}
Here, the search directions $\eta_k \in \mathbb{R}^n$ are computed as $\eta_0 = -\nabla f(x_0)$ and
\begin{equation}
\label{eq:E_direction}
\eta_{k+1} =
-\nabla f(x_{k+1}) + \beta_{k+1}\eta_{k}
\end{equation}
for $k \ge 0$,
where $\beta_{k+1} \in \mathbb{R}$.
For each $k \ge 0$, the step length $t_k > 0$ in~\eqref{eq:E_update} is computed to satisfy, e.g., the Wolfe conditions
\begin{equation}
\label{eq:E_armijo}
f(x_k + t_k \eta_k) \le f(x_k) + c_1 t_k \nabla f(x_k)^T \eta_k
\end{equation}
and
\begin{equation}
\label{eq:E_wolfe}
\nabla f(x_k+t_k\eta_k)^T \eta_k \ge c_2 \nabla f(x_k)^T \eta_k
\end{equation}
with constants $c_1$ and $c_2$ satisfying $0 < c_1 < c_2 < 1$ or the strong Wolfe conditions~\eqref{eq:E_armijo} and
\begin{equation}
\label{eq:E_swolfe}
|\nabla f(x_k+t_k\eta_k)^T \eta_k| \le c_2 |\nabla f(x_k)^T \eta_k|.
\end{equation}
The inequality~\eqref{eq:E_armijo} alone is known as the Armijo condition.
The conditions~\eqref{eq:E_wolfe} and \eqref{eq:E_swolfe} are meaningful if each search direction $\eta_k$ is a descent direction at $x_k$, i.e., $\nabla f(x_k)^T \eta_k < 0$.

The updating rule~\eqref{eq:E_update} is a standard strategy in line search algorithms for optimization problems in Euclidean spaces, and the formula~\eqref{eq:E_direction} for search directions characterizes the CG methods.
The computation of real values $\beta_{k+1}$ is crucial for the performance of CG methods.
Considering the linear CG method where $f(x) = x^T A x / 2 - b^T x$, $\beta_{k+1}$ in~\eqref{eq:E_direction} is uniquely determined for each $k \ge 0$ as $\beta_{k+1} = \| \nabla f(x_{k+1})\|_2^2 / \| \nabla f(x_{k})\|_2^2$ so that the search directions are mutually $A$-conjugate, i.e., $\eta_k^T A \eta_l = 0$ for $0 \le k < l \le n-1$.
On the other hand, several types of $\beta_{k+1}$ have been studied for nonlinear CG methods.
Defining $g_k := \nabla f(x_k)$ and $y_k := g_k - g_{k-1}$, the following six types of $\beta_{k+1}$ are considered as standard ones:
\begin{alignat}{3}
\beta_{k+1}^{\FR} &= \frac{\|g_{k+1}\|_2^2}{\|g_{k}\|_2^2},
&\qquad
\beta_{k+1}^{\DY} &= \frac{\|g_{k+1}\|_2^2}{y_{k+1}^T \eta_k},
&\qquad
\beta_{k+1}^{\CD} &= \frac{\|g_{k+1}\|_2^2}{-g_{k}^T \eta_{k}},\notag
\\
\beta_{k+1}^{\PRP} &= \frac{g_{k+1}^T y_{k+1}}{\|g_{k}\|_2^2},
&\qquad
\beta_{k+1}^{\HS} &= \frac{g_{k+1}^T y_{k+1}}{y_{k+1}^T \eta_k},
&\qquad
\beta_{k+1}^{\LS} &= \frac{g_{k+1}^T y_{k+1}}{-g_{k}^T \eta_{k}}.
\label{eq:E_beta}
\end{alignat}
They were proposed by Fletcher and Reeves~\cite{fletcher1964function}, Dai and Yuan~\cite{dai1999nonlinear}, Fletcher~\cite{fletcher2013practical}, Polak and Ribi\`{e}re~\cite{polak1969note} and Polyak~\cite{polyak1969conjugate}, Hestenes and Stiefel~\cite{hestenes1952methods}, and Liu and Storey~\cite{liu1991efficient}, respectively.
Here, ``CD" in $\beta_{k+1}^{\CD}$ represents ``conjugate descent."
In the linear CG method, all the above six types of $\beta_{k+1}$ mathematically coincide.
However, regarding the nonlinear CG methods, they do not necessarily take the same value, and how to compute $\beta_{k+1}$ affects the performance of the algorithms.
Each formula has been studied separately.
Although we do not describe all the existing CG methods here, various other types of $\beta_{k+1}$ have also been examined.
An important example is $\beta^{\HZ}_{k+1}$ by Hager and Zhang~\cite{hager2005new}, whose Riemannian version is studied in~\cite{sakai2021sufficient}.

\subsection{Review of and discussion on the existing Riemannian CG methods}
In this subsection, we review the existing R-CG methods for solving \cref{prob:gen} on a Riemannian manifold $\M$,
adding some examples to clarify the issues to be resolved in this paper.

In the R-CG methods, the line search strategy~\eqref{eq:E_update} in Euclidean spaces is generalized as
\begin{equation}
\label{eq:M_update}
x_{k+1} = R_{x_k}(t_k\eta_k),
\end{equation}
where $R \colon T\M \to \M$ is a retraction on $\M$ and the search direction $\eta_k$ is in $T_{x_k} \M$.
This formula indicates that, given $x_k \in \M$, we find the subsequent point $x_{k+1}$ on the curve $\gamma_k(t) := R_{x_k}(t\eta_k)$ with $\gamma_k(0) = R_{x_k}(0) = x_k$ and $\dot \gamma_k(0) = \D R_{x_k}(0)[\eta_k] = \eta_k$.
The update formula~\eqref{eq:M_update} is a standard approach also taken in other Riemannian optimization algorithms.
Step lengths $t_k$ should be chosen to satisfy some conditions such as the Riemannian version of the Wolfe conditions~\eqref{eq:E_armijo} and~\eqref{eq:E_wolfe}, or strong Wolfe conditions~\eqref{eq:E_armijo} and~\eqref{eq:E_swolfe}.
In the remainder of this section, we assume that $\eta_k$ is a descent direction, i.e., $\langle \grad f(x_k), \eta_k\rangle_{x_k} < 0$.
Defining a one-variable function $\phi_k \colon \mathbb{R} \to \mathbb{R}$ as $\phi_k(t) := f(R_{x_k}(t\eta_k))$,
we obtain $\phi'_k(0) = \langle \grad f(x_k), \eta_k\rangle_{x_k} < 0$.
Therefore, conditions~\eqref{eq:E_armijo}--\eqref{eq:E_swolfe} can be rewritten, with $R_x(\eta) := x+\eta \in \mathbb{R}^n$, as
\begin{equation}
\label{eq:phi_armijo}
\phi_k(t_k) \le \phi_k(0) + c_1t_k\phi'_k(0),
\end{equation}
\begin{equation}
\label{eq:phi_wolfe}
\phi'_k(t_k) \ge c_2 \phi'_k(0),
\end{equation}
and
\begin{equation}
\label{eq:phi_swolfe}
|\phi'_k(t_k)| \le c_2 |\phi'_k(0)|,
\end{equation}
respectively.
Then, the Riemannian version of~\eqref{eq:E_armijo}--\eqref{eq:E_swolfe} can be obtained by rewriting~\eqref{eq:phi_armijo}--\eqref{eq:phi_swolfe}, with a general retraction $R$, as
\begin{equation}
\label{eq:M_armijo}
f(R_{x_k}(t_k\eta_k)) \le f(x_k) + c_1 t_k \langle \grad f(x_k), \eta_k\rangle_{x_k},
\end{equation}
\begin{equation}
\label{eq:M_wolfe}
\langle \grad f(R_{x_k}(t_k\eta_k)), \D R_{x_k}(t_k\eta_k)[\eta_k]\rangle_{R_{x_k}(t_k\eta_k)} \ge c_2 \langle \grad f(x_k), \eta_k\rangle_{x_k},
\end{equation}
and
\begin{equation}
\label{eq:M_swolfe}
|\langle \grad f(R_{x_k}(t_k\eta_k)), \D R_{x_k}(t_k\eta_k)[\eta_k]\rangle_{R_{x_k}(t_k\eta_k)}|
 \le c_2 |\langle \grad f(x_k), \eta_k\rangle_{x_k}|,
\end{equation}
respectively, where $c_1$ and $c_2$ are again constants satisfying $0 < c_1 < c_2 < 1$.
We call~\eqref{eq:M_armijo} the (Riemannian) Armijo condition, \eqref{eq:M_armijo} and~\eqref{eq:M_wolfe} the (Riemannian) Wolfe conditions, and~\eqref{eq:M_armijo} and~\eqref{eq:M_swolfe} the (Riemannian) strong Wolfe conditions.
Furthermore, we define the (Riemannian) generalized Wolfe conditions as~\eqref{eq:M_armijo} and
\begin{align}
c_2 \langle \grad f(x_k), \eta_k\rangle_{x_k} & \le \langle \grad f(R_{x_k}(t_k\eta_k)), \D R_{x_k}(t_k\eta_k)[\eta_k]\rangle_{R_{x_k}(t_k\eta_k)}\notag\\
& \le - c_3 \langle \grad f(x_k), \eta_k\rangle_{x_k},
\label{eq:M_gwolfe}
\end{align}
where $c_3 \geq 0$ is a constant.
Theoretically, if $c_3 = c_2$, then the generalized Wolfe conditions are equivalent to the strong Wolfe conditions.
However, in some practical cases, $c_3$ can be much larger than $c_2$; thus, it is not so restrictive. If $c_3$ is sufficiently large, the generalized Wolfe conditions are close to the Wolfe conditions.
On the other hand, if $c_3 = 0$, then the generalized Wolfe conditions are stricter than the strong Wolfe conditions.

We proceed to generalizing the computation of the search directions~\eqref{eq:E_direction} in the CG methods to the Riemannian case.
The search direction at the initial point $x_0 \in \M$ is naturally determined, i.e., $\eta_0 = -\grad f(x_0)$.
At the right-hand side of~\eqref{eq:E_direction}, $-\nabla f(x_{k+1}) \in \mathbb{R}^n$ is generalized to the negative Riemannian gradient $-\grad f(x_{k+1}) \in T_{x_{k+1}} \M$ on $\M$, whereas $\beta_{k+1} \in \mathbb{R}$ and $\eta_{k}\in T_{x_{k}} \M$ on $\M$.
Consequently, $-\grad f(x_{k+1})$ and $\beta_{k+1} \eta_{k}$ belong to the distinct tangent spaces $T_{x_{k+1}} \M$ and $T_{x_k} \M$; therefore, they cannot be added together.
To resolve this issue, several approaches were considered in the literature.

In~\cite{edelman1998geometry,lichnewsky1979une,smith1994optimization}, R-CG methods were discussed in which $\eta_k \in T_{x_k} \M$ is parallel translated to $T_{x_{k+1}}\M$ to compute $\eta_{k+1}$, i.e., $\eta_{k+1}$ is computed as
\begin{equation}
\label{eq:parallel}
\eta_{k+1} = -\grad f(x_{k+1}) + \beta_{k+1} \P^{1 \gets 0}_{\gamma_k}(\eta_k),
\end{equation}
where $\P^{1 \gets 0}_{\gamma_k} \colon T_{x_k} \M \to T_{x_{k+1}} \M$ is the parallel translation along the geodesic $\gamma_k$ connecting $x_k$ and $x_{k+1}$ as $\gamma_k(0) = x_k$ and $\gamma_k(1) = x_{k+1}$.
However, in some cases, the parallel translation is numerically impractical.
For example, no closed form for the parallel translation along the geodesic on the Stiefel manifold is known.

In~\cite{AbsMahSep2008}, the concept of a more general map, called a vector transport, was proposed.
In the following definition, $T\M \oplus T\M := \{(\xi, \eta) \mid \xi, \eta \in T_x \M, \ x \in \M\}$ is the Whitney sum.
\begin{definition}
\label{def:VT}
A map $\mathcal{T} \colon T\M \oplus T\M \to T\M$ is called a vector transport on $\M$ if it satisfies the following conditions:
\begin{enumerate}
\item
There exists a retraction $R$ on $\M$ such that $\mathcal{T}_{\eta}(\xi) \in T_{R_{x}(\eta)} \M$ for all $x \in \M$ and $\xi, \eta \in T_x \M$.
\item
For any $x \in \M$ and $\xi \in T_x \M$, $\mathcal{T}_{0_x}(\xi) = \xi$ holds, where $0_x$ is the zero vector in $T_x \M$, i.e., $\mathcal{T}_{0_x}$ is the identity map.
\item
For any $a, b \in \mathbb{R}$, $x \in \M$, and $\xi, \eta, \zeta \in T_x \M$, $\mathcal{T}_{\eta}(a\xi + b\zeta) = a\mathcal{T}_{\eta}(\xi) + b\mathcal{T}_{\eta}(\xi)$ holds, i.e., $\mathcal{T}_{\eta}$ is a linear map from $T_x \M$ to $T_{R_{x}(\eta)} \M$.
\end{enumerate}
\end{definition}
Note that a map $\mathcal{T}$ defined by $\mathcal{T}_{\eta}(\xi) := \P^{1 \gets 0}_{\gamma_{x,\eta}}(\xi)$ is a vector transport, where $\P^{1 \gets 0}_{\gamma_{x,\eta}}$ is the parallel translation along the geodesic $\gamma_{x,\eta}(t) := \Exp_x(t\eta)$ connecting $\gamma_{x,\eta}(0) = x$ and $\gamma_{x,\eta}(1) = \Exp_x(\eta)$ with the exponential map $\Exp$ as a retraction.

Using a general vector transport $\mathcal{T}$ on $\M$, the formula~\eqref{eq:parallel} is generalized to
\begin{equation}
\label{eq:vt}
\eta_{k+1} = -\grad f(x_{k+1}) + \beta_{k+1} \mathcal{T}_{t_k\eta_k}(\eta_k).
\end{equation}
Note that the right-hand side is well-defined because the first condition in \cref{def:VT} ensures $\mathcal{T}_{t_k\eta_k}(\eta_k) \in T_{R_{x_k}(t_k\eta_k)}\M = T_{x_{k+1}}\M$.
The formula~\eqref{eq:vt} is more general than~\eqref{eq:parallel} since~\eqref{eq:vt} includes~\eqref{eq:parallel} as a special case.
However, careful consideration is required to make the resultant R-CG method work appropriately.
A specific vector transport and its modified version (called a scaled vector transport) have been studied to be utilized in R-CG methods as follows.

In~\cite{ring2012optimization}, Ring and Wirth analyzed the R-CG method~\eqref{eq:M_update} and~\eqref{eq:vt} with a specific type of $\beta_{k+1}$ defined by $\beta_{k+1} = \|{\grad f(x_{k+1})}\|_{x_{k+1}}^2 / \|{\grad f(x_{k})}\|_{x_{k}}^2$, which is a natural generalization of $\beta^{\FR}_{k+1}$ in~\eqref{eq:E_beta}.
They proved the global convergence property of this type of R-CG method with the differentiated retraction $\mathcal{T}^R$ as a vector transport $\mathcal{T}$ in \eqref{eq:vt}, i.e.,
\begin{equation}
\label{eq:TR}
\eta_{k+1} = -\grad f(x_{k+1}) + \beta_{k+1} \mathcal{T}^R_{t_k\eta_k}(\eta_k)
\end{equation}
with
\begin{equation}
\label{eq:defTR}
\mathcal{T}^R_{\eta}(\xi) := \D R_x(\eta)[\xi], \qquad \xi, \, \eta \in T_x \M, \quad x \in \M,
\end{equation}
assuming the inequality
\begin{equation}
\label{eq:RWineq}
\|\mathcal{T}^R_{t_k\eta_k}(\eta_k)\|_{x_{k+1}} \le \|\eta_k\|_{x_k}, \qquad k \ge 0.
\end{equation}
However, this inequality does not necessarily hold, even in very natural situations as shown by the following example.
\begin{example}
We consider the following QR-based retraction $R$ on the Stiefel manifold $\St(p,n)$ with $p \le n$:
\begin{equation}
\label{eq:QR}
R_X(\eta) = \qf(X+\eta), \qquad \eta \in T_X \St(p,n), \quad X \in \St (p,n),
\end{equation}
where $T_X \St(p,n) = \{\xi \in \mathbb{R}^{n \times p} \mid X^T \xi + \xi^T X = 0\}$ and $\qf(\cdot)$ returns the Q-factor of the QR decomposition of the full-rank matrix in parentheses, i.e., if $A \in \mathbb{R}^{n \times p}$ is of full-rank\footnote{In~\eqref{eq:QR}, $X+\eta$ is always of full-rank because $\Rank(X+\eta) = \Rank((X+\eta)^T(X+\eta)) = \Rank(I_p + \eta^T \eta) = p$. Note that $X^T \eta + \eta^T X = 0$ since $\eta$ belongs to $T_X \St(p,n)$. Furthemore, $I_p + \eta^T \eta$ is positive definite because $z^T (I_p+\eta^T \eta)z = \|z\|_2^2 + \|\eta z\|_2^2 > 0$ for any nonzero vector $z \in \mathbb{R}^p$. Therefore, it is invertible.} and is decomposed as $A = QR$ with $Q \in \St(p,n)$ and $R$ being an upper triangular matrix with positive diagonal elements (such a decomposition is shown to be unique), then $\qf(A) = Q$.
For this retraction $R$, the differentiated retraction $\mathscr{T}^R$ is computed as~\cite{AbsMahSep2008}
\begin{equation}
\label{eq:VT_St}
\mathcal{T}^R_{\eta}(\xi) := \D R_X(\eta)[\xi] = X_+\rho_{\skew}(X_+^T\xi R_+^{-1}) + (I_n - X_+X_+^T)\xi R_+^{-1} \end{equation}
for $X \in \St(p,n)$ and $\xi, \eta \in T_X \St(p,n)$,
where $X+\eta = X_+R_+$ is the QR decomposition of $X+\eta$ (i.e., $X_+ = \qf(X+\eta)$ and $R_+ = X_+^T(X+\eta)$) and $\rho_{\skew}(\cdot)$ returns the skew-symmetric matrix that has the same size and strict lower part as those of the matrix in parentheses.

When $n \geq p \geq 2$, inequality~\eqref{eq:RWineq} does not necessarily hold.
As an example, we consider the Stiefel manifold $\St(p,n)$ with $n = p = 3$, which is reduced to the orthogonal group $\mathcal{O}(p) = \mathcal{O}(3)$, as a Riemannian submanifold of the Euclidean space $\mathbb{R}^{p \times p}$, i.e., the Riemannian metric on $\St(p,n)$ is defined as $\langle \xi, \eta\rangle_X = \tr(\xi^T \eta)$ for $X \in \St(p,n)$ and $\xi, \eta \in T_X\St(p,n)$.
Assume that $X_k = I_3 \in \mathcal{O}(3)$, $\eta_k = \begin{pmatrix}0 & -1 & -1 \\ 1 & 0 & -1\\ 1 & 1 & 0\end{pmatrix} \in T_{X_k} \mathcal{O}(3) = T_{I_3}\mathcal{O}(3) = \{Y \in \mathbb{R}^{3 \times 3} \mid Y + Y^T = 0\}$, and $t_k = 0.1$.
Then, from~\eqref{eq:VT_St}, a tedious calculation yields $\|\mathcal{T}^R_{t_k\eta_k}(\eta_k)\|_{R_{X_k}(t_k\eta_k)} \approx 2.47 > \sqrt{6} = \|\eta_k\|_{X_k}$,
which violates~\eqref{eq:RWineq}.\footnote{An exact calculation shows that $\|\mathcal{T}^R_{t_k\eta_k}(\eta_k)\|_{R_{X_k}(t_k\eta_k)} = 200\sqrt{42849907}/530553$.}
\end{example}

While inequality~\eqref{eq:RWineq} is important for the global convergence property, the linearity of a vector transport is not crucial in R-CG methods.
Based on this important observation, Sato and Iwai~\cite{sato2015new} proposed the notion of a scaled vector transport $\mathcal{T}^{(0)}$ associated with a vector transport $\mathcal{T}$ defined as
$\mathcal{T}^{(0)}_{\eta}(\xi) := (\|\xi\|_{x}/\|\mathcal{T}_{\eta}(\xi)\|_{R_x(\eta)})\mathcal{T}_{\eta}(\xi)$ for $x \in \M$ and $\xi, \eta \in T_x \M$.
They also proposed a strategy where, in~\eqref{eq:TR}, the scaled vector transport $\mathcal{T}^{(0)}$ associated with the differentiated retraction $\mathcal{T}^R$ is used instead of $\mathcal{T}^R$ only when inequality~\eqref{eq:RWineq} is violated, and $\mathcal{T}^R$ itself without scaling is used if otherwise.
This approach can be regarded as
\begin{equation}
\label{eq:RCG_direction_scaled}
\eta_{k+1} = -\grad f(x_{k+1}) + \beta_{k+1} s_k \mathcal{T}^R_{t_k\eta_k}(\eta_k)
\end{equation}
with the scaling parameter $s_k := \min\{1, {\|\eta_k\|_{x_k}}/{\|\mathcal{T}^R_{t_k\eta_k}(\eta_k)\|_{R_{x_k}(t_k\eta_k)}}\} > 0$, where $t_k\eta_k \neq 0$.
Note that $\mathcal{T}^{(0)}$ is not a vector transport because $\mathcal{T}^{(0)}_{\eta} \colon T_x \M \to T_{R_{x}(\eta)}\M$ for $x \in \M$ and $\eta \in T_x \M$ is not a linear map.
Considering the same framework, Sato~\cite{sato2016dai}, Sakai and Iiduka~\cite{sakai2020hybrid}, and Sakai and Iiduka~\cite{sakai2021sufficient} analyzed the Dai--Yuan-type, some hybrid $\beta_{k+1}$-based, and Hager--Zhang-type of R-CG methods, respectively.

Recently, Zhu and Sato~\cite{zhu2020riemannian} proposed a completely different approach from~\eqref{eq:vt} or~\eqref{eq:TR}.
Their algorithm uses an additional retraction $R^{\bw}$, where ``\bw" represents ``backward" and $R^{\bw}$ can be the same as or different from $R$ in~\eqref{eq:M_update}.
Specifically, given $\eta_k \in T_{x_k} \M$ and $x_{k+1} \in \M$, the search direction $\eta_{k+1}$ at $x_{k+1}$ is computed as
\begin{equation}
\label{eq:RCG_direction_invret}
\eta_{k+1} = -\grad f(x_{k+1}) - \beta_{k+1} s_k t_k^{-1}(R_{x_{k+1}}^{\bw})^{-1}(x_k)
\end{equation}
with the scaling parameter $s_k := \min\big\{1, \|\eta_k\|_{x_k}\big/\|t_k^{-1}\big(R_{x_{k+1}}^{{\textrm{bw}}}\big)^{-1}(x_k)\|_{x_{k+1}}\big\}$.
This indicates that the quantity $-t_k^{-1}(R^{\bw}_{x_{k+1}})^{-1}(x_k)$ is used in~\eqref{eq:RCG_direction_invret} instead of $\mathcal{T}^R_{t_k\eta_k}(\eta_k)$ in~\eqref{eq:RCG_direction_scaled}.
In~\cite{zhu2020riemannian}, the FR- and DY-types of R-CG methods with inverse retraction are analyzed.
Furthermore, the inverse retraction is easily computed in some specific cases (e.g., when $R^{\bw}$ is the orthographic retraction).
Another example of inverse retraction $(R^{\bw})^{-1}$ on the manifold of symmetric positive definite matrices, without the necessity of knowing the explicit expression of $R^{\bw}$, is found in~\cite{goto2021approximated}.
These computational advantages make the R-CG methods with inverse retraction practical.
Formula~\eqref{eq:RCG_direction_invret} is a vector transport-free approach, which demonstrates that other approaches can also be possible for the R-CG methods, leading to the idea of the general framework proposed in the subsequent section.

\section{New general framework of Riemannian CG methods}
\label{sec:3}
In this section, we propose a new framework of R-CG methods, which contains all the existing R-CG methods as special cases.
Furthermore, we generalize standard formulas for $\beta_{k+1}$ in~\eqref{eq:E_beta} to the Riemannian case.
Some conditions for step lengths in the proposed framework are also introduced.

\subsection{Algorithm}
We propose a new general framework of R-CG methods in which we use a general map $\mathscr{T}^{(k)} \colon T_{x_k} \M \to T_{x_{k+1}} \M$ to transport $\eta_k \in T_{x_k}\M$ to $T_{x_{k+1}} \M$, i.e., the search direction $\eta_{k+1}$ is computed as
\begin{equation}
    \label{eq:RCG_direction}
\eta_{k+1} = -\grad f(x_{k+1}) + \beta_{k+1} s_k \T^{(k)}(\eta_k),
\end{equation}
where $s_k$ is a scaling parameter satisfying
\begin{equation}
\label{eq:sk}
    0 < s_k \le \min\left\{1, \dfrac{\|\eta_k\|_{x_k}}{ \|\T^{(k)}(\eta_k)\|_{x_{k+1}}}\right\}.
\end{equation}
We summarize the proposed framework of the R-CG methods as \cref{alg:RCGgeneral}.
\begin{algorithm}[H]
\caption{General framework of the Riemannian conjugate gradient methods for \cref{prob:gen} on Riemannian manifold $\M$ with retraction $R$}
\label{alg:RCGgeneral}
\begin{algorithmic}[1]
\STATE Choose an initial point $x_0\in \M$ and set $\eta_0:=-\grad f(x_0)$ and $k := 0$.
\WHILE{$\grad f(x_k)\neq 0$}
\STATE Compute a step length $t_k > 0$ and $x_{k+1}:=R_{x_k}(t_k\eta_k)$.
\STATE Compute $\beta_{k+1} \in \mathbb{R}$.
\STATE
Compute a search direction as $\eta_{k+1}:=-\grad f(x_{k+1})+\beta_{k+1}{s_k\mathscr{T}^{(k)}\left(\eta_k\right)}$ with a map $\T^{(k)} \colon T_{x_k} \M \to T_{x_{k+1}} \M$ and a scaling parameter $s_k \in \mathbb{R}$ satisfying $0 < s_k \le \min\left\{1, \|\eta_k\|_{x_k}/ \|\T^{(k)}(\eta_k)\|_{x_{k+1}}\right\}$.
\STATE $k:=k+1$.
\ENDWHILE
\end{algorithmic}
\end{algorithm}

We need to clarify what conditions $\mathscr{T}^{(k)}$ should satisfy, how to compute $\beta_k$, and how step lengths $t_k$ should be chosen.
We address these in the subsequent subsections.

\subsection{Map \boldmath$\mathscr{T}^{(k)}$ and scaling parameter $s_k$}
The map $\mathscr{T}^{(k)}$ in \cref{alg:RCGgeneral} can be any map such that it appropriately transports $\eta_k \in T_{x_k}\M$ to $T_{x_{k+1}}\M$.
Several conditions used in convergence analyses in \cref{sec:5} are discussed at the end of this subsection. 
An important feature of \cref{alg:RCGgeneral} is that we do not necessarily require $\mathscr{T}^{(k)}$ to be based on a vector transport.
Furthermore, $\mathscr{T}^{(k)}$ is not necessarily a linear map.
Therefore, the R-CG method with inverse retraction introduced in~\eqref{eq:RCG_direction_invret} is also contained in this framework as specifically explained in \cref{ex:T}.

Furthermore, any inequality corresponding to \eqref{eq:RWineq} is not required in terms of $\T^{(k)}$.
Instead, the scaling parameter $s_k \in (0, \min\{1, \|\eta_k\|_{x_k} / \|\T^{(k)}(\eta_k)\|_{x_{k+1}}\}]$ plays a role to ensure a similar inequality $\|s_k\T^{(k)}(\eta_k)\|_{x_{k+1}} \le \|\eta_k\|_{x_k}$ for all $k \ge 0$.

\begin{example}
\label{ex:T}
In \cref{alg:RCGgeneral}, we know several choices of $\mathscr{T}^{(k)}$ since this algorithm includes all the R-CG methods introduced in \cref{sec:2}.
If we set $s_k := 1$ and $\mathscr{T}^{(k)}(\eta_k) := \P_{\gamma_k}^{1 \gets 0}(\eta_k)$ in \cref{alg:RCGgeneral} with parallel translation $\P_{\gamma_k}$ along the geodesic $\gamma_k$ connecting $x_k$ and $x_{k+1}$, then~\eqref{eq:RCG_direction} reduces to~\eqref{eq:parallel}.
Here, we can take $s_k = 1$ because the parallel translation is isometric, i.e., we have $\|\P_{\gamma_k}^{1\gets 0}(\eta_k)\|_{x_{k+1}} = \|\eta_k\|_{x_k}$.
If we set $s_k := \min\{1, {\|\eta_k\|_{x_k}}/{\|\mathcal{T}^R_{t_k\eta_k}(\eta_k)\|_{x_{k+1}}}\}$ and $\mathscr{T}^{(k)}(\eta_k) := \mathcal{T}^R_{t_k\eta_k}(\eta_k)$ with the differentiated retraction $\mathcal{T}^R$, then \eqref{eq:RCG_direction} reduces to~\eqref{eq:RCG_direction_scaled}.
Furthermore, if we set $s_k := \min\big\{1, \|\eta_k\|_{x_k}\big/\|t_k^{-1}\big(R_{x_{k+1}}^{{\textrm{bw}}}\big)^{-1}(x_k)\|_{x_{k+1}}\big\}$ and $\mathscr{T}^{(k)}(\eta_k) := -t_k^{-1}(R^{\bw}_{R_{x_k}(t_k\eta_k)})^{-1}(x_k)$,
then \eqref{eq:RCG_direction} reduces to~\eqref{eq:RCG_direction_invret}.
For these three examples, the chosen $s_k$ can be uniformly written as $s_k = \min\{1, \|\eta_k\|_{x_k} / \|\T^{(k)}(\eta_k)\|_{x_{k+1}}\}$.

In~\cite{sato2021riemannian}, Sato proposed a prototype of \cref{alg:RCGgeneral}, where $s_k$ is any real value satisfying~\eqref{eq:sk} and $\T^{(k)}(\eta_k) := \mathcal{T}^R_{t_k\eta_k}(\eta_k)$.
However, \cref{alg:RCGgeneral} is more general than the algorithm in~\cite{sato2021riemannian} because $\T^{(k)}$ is not restricted to the differentiated retraction-based map.
\end{example}

\cref{alg:RCGgeneral} covers a wider class of R-CG methods than the existing ones partly because we do not limit the scaling parameter $s_k > 0$ to be a specific form such as $s_k = \min\{1, \|\eta_k\|_{x_k} / \|\T^{(k)}(\eta_k)\|_{x_{k+1}}\}$.
Moreover, it contains R-CG methods with $\T^{(k)}$ that have not been discussed in the literature.
An example is to use a vector transport based on the orthogonal projection to the tangent spaces, which we detail for the sphere case in \cref{ex:P}.

Subsequently, we introduce two conditions~\eqref{eq:TkDR1} and~\eqref{eq:TkDR2} on $\T^{(k)}$, which are used in global convergence analyses of R-CG methods in \cref{sec:5}.
\begin{assumption}
\label{assump:Tk}
For maps $\T^{(k)}$ in \cref{alg:RCGgeneral}, there exist $C \geq 0$ and index sets $K_1 \subset \mathbb{N}$ and $K_2 = \mathbb{N} - K_1$ such that
\begin{equation}
\label{eq:TkDR1}
    \|\mathscr{T}^{(k)}(\eta_k) - \D R_{x_k}(t_k\eta_k)[\eta_k]\|_{x_{k+1}} \le Ct_k\|\eta_k\|_{x_k}^2, \qquad k \in K_1
\end{equation}
and
\begin{equation}
\label{eq:TkDR2}
    \|\mathscr{T}^{(k)}(\eta_k) - \D R_{x_k}(t_k\eta_k)[\eta_k]\|_{x_{k+1}} \le C(t_k+t_k^2)\|\eta_k\|_{x_k}^2,\qquad k \in K_2
\end{equation}
hold, where $\mathbb{N}$ is the set of nonnegative integers.
\end{assumption}

\begin{remark}
In fact, for any $k \ge 0$, the inequality in~\eqref{eq:TkDR2} is weaker than that in~\eqref{eq:TkDR1}.
Therefore, \cref{assump:Tk} is equivalent to the condition that there exists $C \geq 0$ such that $\|\mathscr{T}^{(k)}(\eta_k) - \D R_{x_k}(t_k\eta_k)[\eta_k]\|_{x_{k+1}} \le C(t_k+t_k^2)\|\eta_k\|_{x_k}^2$ holds for all $k \ge 0$.
We write \cref{assump:Tk} with a stricter inequality~\eqref{eq:TkDR1} because we can weaken the condition imposed on the step length $t_k$ when $\T^{(k)}$ satisfies~\eqref{eq:TkDR1} in Zoutendijk's theorem (\cref{thm:zouten}).
\end{remark}

\cref{assump:Tk} requires that $\T^{(k)}$ is not so far from the differentiated retraction $\mathcal{T}^R$.
The R-CG methods with the differentiated retraction trivially satisfies the assumption, especially~\eqref{eq:TkDR1} with $C = 0$ and $K_1 = \mathbb{N}$.
Furthermore, in~\cite{zhu2020riemannian}, it is discussed that this assumption, especially~\eqref{eq:TkDR2} with $K_2 = \mathbb{N}$, is also natural for the R-CG methods with inverse retraction.

The proposed R-CG methods (\cref{alg:RCGgeneral}) will be analyzed in \cref{sec:5} with \cref{assump:Tk}.
We realize that R-CG methods with some vector transports $\mathcal{T}$ that have not been analyzed yet also have convergence properties if $\T^{(k)}$ defined by $\mathcal{T}$ satisfies \cref{assump:Tk}.
Examples of such vector transports are shown as follows.

\begin{example}
\label{ex:P}
Consider the sphere $S^{n-1} := \{x \in \mathbb{R}^n \mid x^T x = 1\}$ with a retraction $R$ and Riemannian metric $\langle \cdot, \cdot\rangle$ defined as $R_x(\eta) = (x+\eta) / \|x+\eta\|_2$ and $\langle \xi, \eta\rangle_x = \xi^T \eta$ for $x \in S^{n-1}$ and $\xi, \eta \in T_x S^{n-1}$, respectively.
Thus, we regard $S^{n-1}$ as a Riemannian submanifold of $\mathbb{R}^n$.
Generally, for a Riemannian submanifold, the orthogonal projections to the tangent spaces define a vector transport~\cite{AbsMahSep2008}.
We can define such a vector transport $\mathcal{T}^P$ on $S^{n-1}$ based on the orthogonal projection as
\begin{equation}
    \mathcal{T}^P_{\eta}(\xi) := P_{R_x(\eta)}(\xi) = (I_n - R_x(\eta)R_x(\eta)^T)\xi = \bigg(I_n - \frac{(x+\eta)(x+\eta)^T}{\|x+\eta\|_2^2}\bigg)\xi,
\end{equation}
for $x \in S^{n-1}$ and $\eta, \xi \in T_x S^{n-1}$, where $P_y(d) = (I_n - yy^T)d$ is the orthogonal projection of $d \in \mathbb{R}^n$ to the tangent space $T_y S^{n-1} = \{z \in \mathbb{R}^n \mid y^T z = 0\}$ at $y \in S^{n-1}$.
This vector transport is typically used in R-CG methods practically (e.g., implemented in Manopt~\cite{boumal2014manopt}); nevertheless, to the author's knowledge, it has not been theoretically discussed in detail in terms of the convergence properties of the R-CG methods.
Therefore, it is meaningful to verify that $\T^{(k)}$ defined by $\mathcal{T}^P$ satisfies \cref{assump:Tk}.

The differentiated retraction $\mathcal{T}^R$ is written as
\begin{equation*}
    \mathcal{T}^R_{\eta}(\xi) := \D R_x(\eta)[\xi] = \frac{1}{\|x+\eta\|_2}\bigg(I_n - \frac{(x+\eta)(x+\eta)^T}{\|x+\eta\|_2^2}\bigg)\xi = \frac{1}{\|x+\eta\|_2}P_{R_x(\eta)}(\xi).
\end{equation*}
Therefore, we can evaluate the difference of the two vector transports, with $\eta = t_k\eta_k$ and $\xi = \eta_k \in T_{x_k} S^{n-1}$, as
\begin{equation}
\label{eq:ex_sphere1}
    \|\mathcal{T}^P_{t_k\eta_k}(\eta_k) - \mathcal{T}^R_{t_k\eta_k}(\eta_k)\|_{R_{x_k}(t_k\eta_k)} = \bigg|1-\frac{1}{\|x_k+t_k\eta_k\|_2}\bigg| \|P_{R_{x_k}(t_k\eta_k)}(\eta_k)\|_{R_{x_k}(t_k\eta_k)}.\\
\end{equation}
Considering that $x_k^T x_k = 1$ and $x_k^T \eta_k = 0$, we obtain $\|x_k+t_k\eta_k\|_2
= \sqrt{1+t_k^2\|\eta_k\|_2^2} > 0$ and
\begin{equation}
\label{eq:ex_sphere2}
    \|P_{R_{x_k}(t_k\eta_k)}(\eta_k)\|_{R_{x_k}(t_k\eta_k)}^2 = \bigg\|\eta_k - \frac{t_k\|\eta_k\|_2^2}{1+t_k^2\|\eta_k\|_2^2}(x_k+t_k\eta_k)\bigg\|_2^2
    =\frac{\|\eta_k\|_2^2}{1+t_k^2\|\eta_k\|_2^2}.
\end{equation}
Here, we define the auxiliary function $h(t) := (1 - 1/\sqrt{1+t^2})/(t\sqrt{1+t^2})$ on $t > 0$.
We can find an upper bound of $h$ as
\begin{equation}
\label{eq:ex_sphere3}
    h(t) = \frac{\sqrt{1+t^2}-1}{t(1+t^2)} = \frac{t}{(1+t^2)(\sqrt{1+t^2}+1)} < \bigg(\bigg(\frac{1}{t}+t\bigg)\cdot 2\bigg)^{-1}
    \leq \frac{1}{4},
\end{equation}
where we used $\sqrt{1+t^2} > 1$ and $t+1/t \geq 2$ from the arithmetic--geometric mean inequality.\footnote{A more tedious calculation provides the strictest upper bound as $h(t) \le 4\sqrt{2/(349+85\sqrt{17})} \approx 0.2139$, which is the maximum value of $h$.}
If $t_k\|\eta_k\|_2 > 0$, combining~\eqref{eq:ex_sphere1}--\eqref{eq:ex_sphere3}, we obtain
\begin{align*}
    \|\mathcal{T}^P_{t_k\eta_k}(\eta_k) - \mathcal{T}^R_{t_k\eta_k}(\eta_k)\|_{R_{x_k}(t_k\eta_k)}
    &= \bigg(1-\frac{1}{\sqrt{1+(t_k\|\eta_k\|_2)^2}}\bigg)\frac{\|\eta_k\|_2}{\sqrt{1+(t_k\|\eta_k\|_2)^2}} \\
    &= h(t_k\|\eta_k\|_2)\cdot t_k\|\eta_k\|_2^2 \le \frac{1}{4}t_k\|\eta_k\|_2^2
    = \frac{1}{4}t_k\|\eta_k\|_{x_k}^2,
\end{align*}
whereas if $t_k \eta_k = 0$, we have $\|\mathcal{T}^P_0(\eta_k) - \mathcal{T}^R_0(\eta_k)\|_{x_k} = \|\eta_k-\eta_k\|_{x_k} = 0$.
Therefore, for the sphere $S^{n-1}$, $\T^{(k)}(\eta_k) := \mathcal{T}^P_{t_k\eta_k}(\eta_k)$ satisfies the condition in \cref{assump:Tk} with $C = 1/4$ and $K_1 = \mathbb{N}$.
\end{example}

\begin{example}
\label{ex:P_Grass}
Consider the Grassmann manifold $\Grass(p,n) \simeq \St(p,n) / \mathcal{O}(p)$ with $p \le n$.
For $X \in \Grass(p,n)$, let $\bar{X} \in \St(p,n)$ denote a representative of $X$ and let $\bar{\eta}$ denote the horizontal lift of $\eta \in T_X\Grass(p,n)$ at $\bar{X}$.
We endow $\Grass(p,n)$ with the Riemannian metric $\langle \xi, \eta\rangle_X := \tr(\bar{\xi}^T\bar{\eta})$ for $\xi, \eta \in T_X \Grass(p,n)$ and the retraction $R$ based on the polar decomposition defined through $\overline{R_X(\eta)} := (\bar{X} + \bar{\eta})(I_p + \bar{\eta}^T\bar{\eta})^{-1/2}$.
For this retraction $R$, similarly to the previous example, the projection-based vector transport $\mathcal{T}^P$ and differentiated retraction $\mathcal{T}^R$ are written as~\cite{AbsMahSep2008,huang2013optimization}
\begin{equation*}
    \overline{\mathcal{T}^P_{\eta}(\xi)} = (I_n - YY^T)\bar{\xi},\qquad
    \overline{\mathcal{T}^R_{\eta}(\xi)} = (I_n -YY^T)\bar{\xi}(Y^T(\bar{X}+\bar{\eta}))^{-1},
\end{equation*}
where $Y = \overline{R_X(\eta)}$.
We can generalize the discussion in \cref{ex:P} to this case.

Here, omitting subscript $k$, putting $\eta^T \eta =: Q\diag(\lambda_1, \lambda_2, \dots, \lambda_p)Q^T$ with $Q \in \mathcal{O}(p)$ and $Z := \overline{R_X(t\eta)}$, and using the auxiliary function $h$ in \cref{ex:P}, we obtain
\begin{align*}
    \|\mathcal{T}^P_{t\eta}(\eta) - \mathcal{T}^R_{t\eta}(\eta)\|_{R_{X}(t\eta)}^2
    &= \|(I_n-ZZ^T)\bar{\eta}(I_p - (Z^T(\bar{X}+t
    \bar{\eta}))^{-1}\|_F^2\\
    &= \tr(\bar{\eta}^T\bar{\eta}(I_p + t^2 \bar{\eta}^T \bar{\eta})^{-1}(I_p - (I_p + t^2 \bar{\eta}^T \bar{\eta})^{-1/2})^2)\\
    &= \sum_{i=1}^p \bigg(1 - \frac{1}{\sqrt{1 + t^2 \lambda_i}}\bigg)^2 \frac{\lambda_i}{1 + t^2 \lambda_i}\\
    &= \sum_{i=1}^p h\big(t\sqrt{\lambda_i}\big)^2t^2\lambda_i^2 \le \bigg(\frac{1}{4}\bigg)^2t^2\tr((\bar{\eta}^T \bar{\eta})^2) \le \bigg(\frac{1}{4}t\|\bar{\eta}\|_{F}^2\bigg)^2,
\end{align*}
implying that $ \|\mathcal{T}^P_{t\eta}(\eta) - \mathcal{T}^R_{t\eta}(\eta)\|_{R_{X}(t\eta)} \le t\|\eta\|_X^2/4$.
Therefore, $\T^{(k)}(\eta_k) := \mathcal{T}^P_{t_k\eta_k}(\eta_k)$ satisfies the condition in \cref{assump:Tk} with $C = 1/4$ and $K_1 = \mathbb{N}$.
\end{example}

\subsection{Computation of \boldmath$\beta_{k+1}$ in R-CG methods}

In \cref{alg:RCGgeneral}, the computation of $\beta_{k+1}$ in each iteration is crucial, and it affects the performance of the R-CG methods.
Some of the six types of $\beta_{k+1}$ in Euclidean CG methods shown in~\eqref{eq:E_beta} have been generalized to the Riemannian case in each R-CG algorithm with a specific choice of $\T^{(k)}$ in the literature.
For example, Smith~\cite{smith1994optimization} and Edelman et al.~\cite{edelman1998geometry} proposed the generalization of $\beta_{k+1}^{\LS}$ and $\beta_{k+1}^{\PRP}$ with parallel translation along the geodesic, respectively.
Ring and Wirth~\cite{ring2012optimization} and Sato and Iwai~\cite{sato2015new} analyzed the generalization of $\beta_{k+1}^{\FR}$ with the (scaled) vector transport defined through the differentiated retraction.
Sato~\cite{sato2016dai} proposed and analyzed the generalization of $\beta_{k+1}^{\DY}$ in the same framework as in~\cite{sato2015new}.
Sakai and Iiduka~\cite{sakai2020hybrid} recently discussed a class of $\beta_{k+1}$ containing a combination of the generalizations of $\beta_{k+1}^{\DY}$ and $\beta_{k+1}^{\HS}$ with the same (scaled) vector transports.
Furthermore, Zhu and Sato~\cite{zhu2020riemannian} proposed and analyzed the generalizations of $\beta_{k+1}^{\FR}$ and $\beta_{k+1}^{\DY}$ with inverse retraction.

Here, we propose the Riemannian versions of the six types of $\beta_{k+1}$, generalized from~\eqref{eq:E_beta} in the Euclidean CG methods.
We put $g_k := \grad f(x_k) \in T_{x_k} \M$.
From~\eqref{eq:E_beta}, we observe that $\beta_{k+1}^{\FR}$, $\beta_{k+1}^{\DY}$, and
$\beta_{k+1}^{\CD}$ have a common numerator $\|g_{k+1}\|_2^2$ in the Euclidean case.
This quantity can be easily and naturally generalized to the Riemannian case as $\|g_{k+1}\|_{x_{k+1}}^2$,
i.e., the Euclidean gradient is replaced with the Riemannian gradient and the Euclidean norm is generalized to the norm in $T_{x_{k+1}}\M$ defined by the Riemannian metric.
On the other hand, $\beta_{k+1}^{\PRP}$, $\beta_{k+1}^{\HS}$, and $\beta_{k+1}^{\LS}$ have the common numerator $g_{k+1}^T y_{k+1}$, where $y_{k+1} := g_{k+1} - g_{k}$.
This is generalized to the Riemannian case on $\M$ by transporting $g_{k} \in T_{x_{k}} \M$ to $T_{x_{k+1}} \M$ using some map $\mathscr{S}^{(k)} \colon T_{x_k}\M \to T_{x_{k+1}} \M$ (which is possibly equal to $\T^{(k)}$) and some scaling parameter $l_k > 0$, and taking the inner product $\langle g_{k+1}, g_{k+1} - l_k\mathscr{S}^{(k)}(g_k)\rangle_{x_k}$ in $T_{x_k} \M$.

Furthermore, $\beta_{k+1}^{\FR}$ and $\beta_{k+1}^{\PRP}$ have the common denominator $\|g_{k}\|_2^2$, which is generalized to $\|g_{k}\|_{x_{k}}^2$, and $\beta_{k+1}^{\CD}$ and $\beta_{k+1}^{\LS}$ have the common denominator $-g_{k}^T \eta_{k}$, which is generalized to $-\langle g_{k}, \eta_{k}\rangle_{x_{k}}$.
Finally, $\beta_{k+1}^{\DY}$ and $\beta_{k+1}^{\HS}$ have the common denominator $y_{k+1}^T\eta_{k}$.
In~\cite{sato2016dai}, where $s_k := \min\{1, {\|\eta_k\|_{x_k}}/{\|\mathcal{T}^R_{t_k\eta_k}(\eta_k)\|_{x_{k+1}}}\}$ and $\T^{(k)}(\eta_k) := \mathcal{T}^R_{t_k\eta_k}(\eta_k)$, the quantity $y_{k+1}^T \eta_{k} = g_{k+1}^T\eta_{k} - g_k^T \eta_{k}$ in the Euclidean case is generalized to the Riemannian case as $\langle g_{k+1}, s_{k}\T^{(k)}(\eta_{k})\rangle_{x_{k+1}} - \langle g_k, \eta_{k}\rangle_{x_{k}}$.
We follow this approach in \cref{alg:RCGgeneral}.

In summary, we obtain the following formulas for the Riemannian version of $\beta_{k+1}$ in~\eqref{eq:RCG_direction}, some of which depend on map $\T^{(k)}$ and $\S^{(k)}$:
\allowdisplaybreaks
\begin{align}
\label{eq:RFR}
\beta_{k+1}^{\RFR} &= \frac{\|{g_{k+1}}\|_{x_{k+1}}^2}{\|{g_k}\|_{x_k}^2},\\
\label{eq:RDY}
\beta_{k+1}^{\RDY} &= \frac{\|{g_{k+1}}\|_{x_{k+1}}^2}{\langle g_{k+1}, s_k\mathscr{T}^{(k)}(\eta_k)\rangle_{x_{k+1}} - \langle g_k, \eta_k\rangle_{x_k}},\\
\label{eq:RCD}
\beta_{k+1}^{\RCD} &= \frac{\|{g_{k+1}}\|_{x_{k+1}}^2}{- \langle g_k, \eta_k\rangle_{x_k}},\\
\label{eq:RPRP}
\beta_{k+1}^{\RPRP} &= \frac{\|{g_{k+1}}\|_{x_{k+1}}^2 - \langle g_{k+1}, l_k\mathscr{S}^{(k)}
(g_k)\rangle_{x_{k+1}}}{\|{g_k}\|_{x_k}^2},\\
\label{eq:RHS}
\beta_{k+1}^{\RHS} &= \frac{\|{g_{k+1}}\|_{x_{k+1}}^2 - \langle g_{k+1}, l_k\mathscr{S}^{(k)}
(g_k)\rangle_{x_{k+1}}}{\langle g_{k+1}, s_k\mathscr{T}^{(k)}(\eta_k)\rangle_{x_{k+1}} - \langle g_k, \eta_k\rangle_{x_k}},\\
\label{eq:RLS}
\beta_{k+1} ^{\RLS} &= \frac{\|{g_{k+1}}\|_{x_{k+1}}^2 - \langle g_{k+1}, l_k\mathscr{S}^{(k)}
(g_k)\rangle_{x_{k+1}}}{- \langle g_k, \eta_k\rangle_{x_k}}.
\end{align}
As explained above, $l_k > 0$ and $\S^{(k)} \colon T_{x_k} \M \to T_{x_{k+1}} \M$ in~\eqref{eq:RPRP}--\eqref{eq:RLS} play similar roles to those of $s_k$ and $\T^{(k)}$, respectively.
However, we do not impose any specific conditions on $l_k$ and $\S^{(k)}$ at this stage.
Practically, it may be desirable that $l_k \S^{(k)}(g_k) \approx g_k$ holds when $t_k \eta_k \approx 0$, indicating when $x_{k+1} \approx x_k$.
The R-CG methods with modified $\beta^{\RPRP}_{k+1}$, $\beta^{\RHS}_{k+1}$, and $\beta^{\RLS}_{k+1}$ will be discussed in detail in \cref{subsec:PRP_HS_LS}.

We can verify that they are all reduced to the corresponding existing $\beta_k$ (if the literature exists) by specifying maps $\mathscr{T}^{(k)}$ and $\mathscr{S}^{(k)}$, such as a vector transport or inverse retraction.
These discussions on generalization of several types of $\beta_{k+1}$ are justified through the convergence analyses in \cref{sec:5}.

\subsection{Step length $t_k$}
In the R-CG methods, the (strong) Wolfe conditions are especially important to guarantee their convergence properties.
Because we introduce a map $\T^{(k)}$ in \cref{alg:RCGgeneral}, we need to slightly modify the conditions.
In this subsection, we assume that the current iteration $x_k \in \M$ and search direction $\eta_k \in T_{x_k} \M$ are given. Further, we assume that $\eta_k$ is a descent direction, i.e., $\langle \grad f(x_k), \eta_k\rangle_{x_k} < 0$.

We revisit conditions \eqref{eq:M_wolfe}--\eqref{eq:M_gwolfe}, which appear in the (strong/generalized) Wolfe conditions.
In these three conditions, the quantity $\D R_{x_k}(t_k\eta_k)[\eta_k]$ is commonly used.
This is written as $\D R_{x_k}(t_k\eta_k)[\eta_k] = \mathcal{T}^R_{t_k\eta_k}[\eta_k]$ for the differentiated retraction $\mathcal{T}^R$ defined as~\eqref{eq:defTR}.
We generalize the (strong/generalized) Wolfe conditions by replacing $\mathcal{T}^R_{t_k\eta_k}(\eta_k)$ with $\T^{(k)}(\eta_k)$.
Specifically, \eqref{eq:M_wolfe}--\eqref{eq:M_gwolfe} are generalized as
\begin{equation}
\label{eq:M_Twolfe}
\langle \grad f(R_{x_k}(t_k\eta_k)), \mathscr{T}^{(k)}(\eta_k)\rangle_{R_{x_k}(t_k\eta_k)} \geq c_2 \langle \grad f(x_k), \eta_k\rangle_{x_k},
\end{equation}
\begin{equation}
\label{eq:M_Tswolfe}
|\langle \grad f(R_{x_k}(t_k\eta_k)), \mathscr{T}^{(k)}(\eta_k)\rangle_{R_{x_k}(t_k\eta_k)}|
 \le c_2 |\langle \grad f(x_k), \eta_k\rangle_{x_k}|,
\end{equation}
and
\begin{align}
c_2 \langle \grad f(x_k), \eta_k\rangle_{x_k} & \le \langle \grad f(R_{x_k}(t_k\eta_k)), \mathscr{T}^{(k)}(\eta_k)\rangle_{R_{x_k}(t_k\eta_k)} \notag\\*
& \le - c_3 \langle \grad f(x_k), \eta_k\rangle_{x_k},
\label{eq:M_Tgwolfe}
\end{align}
respectively.
We define the \emph{$\mathscr{T}^{(k)}$-Wolfe conditions} as~\eqref{eq:M_armijo} and~\eqref{eq:M_Twolfe}, \emph{strong $\mathscr{T}^{(k)}$-Wolfe conditions} as~\eqref{eq:M_armijo} and~\eqref{eq:M_Tswolfe}, and \emph{generalized $\mathscr{T}^{(k)}$-Wolfe conditions} as~\eqref{eq:M_armijo} and~\eqref{eq:M_Tgwolfe}, where $0 < c_1 < c_2 < 1$ and $c_3 \geq 0$.
Note that the scaling parameter $s_k$ in \cref{alg:RCGgeneral} does not appear in these conditions.

Subsequently, we discuss whether $t_k$ satisfying the (strong/generalized) $\T^{(k)}$-Wolfe conditions exists.
It is sufficient to show that $t_k$ satisfying the generalized $\T^{(k)}$-Wolfe conditions~\eqref{eq:M_armijo} and~\eqref{eq:M_Tgwolfe} with $c_3 = 0$ exists because such $t_k$ also satisfies the $\T^{(k)}$-Wolfe conditions~\eqref{eq:M_armijo} and~\eqref{eq:M_Twolfe}, strong $\T^{(k)}$-Wolfe conditions~\eqref{eq:M_armijo} and~\eqref{eq:M_Tswolfe}, and generalized $\T^{(k)}$-Wolfe conditions~\eqref{eq:M_armijo} and~\eqref{eq:M_Tgwolfe} with any $c_3 \ge 0$.

If $\T^{(k)}(\eta_k) := \D R_{x_k}(t_k\eta_k)[\eta_k]$, then the (strong/generalized) $\T^{(k)}$-Wolfe conditions are reduced to the Riemannian (strong/ generalized) Wolfe conditions, i.e., \eqref{eq:M_Twolfe}--\eqref{eq:M_Tgwolfe} are reduced to~\eqref{eq:M_wolfe}--\eqref{eq:M_gwolfe}, respectively.
In particular, the generalized $\T^{(k)}$-Wolfe conditions~\eqref{eq:M_armijo} and~\eqref{eq:M_Tgwolfe} with $c_3 = 0$ in this case are rewritten as $\phi_k(t_k) \le \phi_k(0) + c_1 t_k \phi_k'(0)$ and $c_2 \phi_k'(0) \le \phi_k'(t_k) \le 0$ by defining $\phi_k(t) := f(R_{x_k}(t\eta_k))$.
Then, we can prove that $t_k > 0$ satisfying the two inequalities exists, similar to those in the Euclidean case.
A complete proof for the Riemannian case is found in Proposition 3.5 of~\cite{sato2021riemannian}.
If $\T^{(k)}(\eta_k) := -t_k^{-1}(R^{\bw}_{R_{x_k}(t_k\eta_k)})^{-1}(x_k)$, then the study on R-CG methods with inverse retraction~\cite{zhu2020riemannian} reveals that there exists $t_k$ satisfying the generalized $\T^{(k)}$-Wolfe conditions~\eqref{eq:M_armijo} and~\eqref{eq:M_Tgwolfe} with $c_3 = 0$.

\section{Assumptions and Zoutendijk's theorem}
\label{sec:4}

In this section, we discuss and summarize assumptions required for guaranteeing the global convergence of the R-CG methods.
Although the proposed framework (\cref{alg:RCGgeneral}) is quite general, it is important to clarify the conditions with which the R-CG methods appropriately work.
We have already discussed the conditions for $\T^{(k)}$ and $t_k$ in \cref{alg:RCGgeneral} in \cref{sec:3}.
In \cref{subsec:assumption}, we state the conditions imposed on \cref{prob:gen}.
Furthermore, we extend (the Riemannian version of) Zoutendijk's theorem to a theorem (\cref{thm:zouten}) in the framework of \cref{alg:RCGgeneral}.

\subsection{Assumptions for retraction and objective function}
\label{subsec:assumption}
We assume the following condition on the objective function $f$.
\begin{assumption}
\label{assump:RCG}
The Riemannian manifold $\M$ in \cref{prob:gen} is endowed with a retraction $R \colon T\M \to \M$.
The objective function $f$ in \cref{prob:gen} is of class $C^1$, bounded below on $\M$, i.e., there exists a constant $f_* \in \mathbb{R}$ such that $f(x) \geq f_*$ for all $x \in \M$, and satisfies the following condition:
\begin{align}
&\text{There exists a constant $L > 0$ such that, for all $x \in \M$, $\eta \in T_x \M$ with}\notag\\
&\text{$\|\eta\|_x = 1$, and $t \ge 0$,
it holds $|\D (f \circ R_x)(t\eta)[\eta] - \D (f \circ R_x)(0)[\eta]| \le Lt$.} \label{eq:Lip}
\end{align}
Furthermore, the norm of the gradient of $f$ is upper bounded on the sublevel set $\{x \in \M \mid f(x) \le f(x_0)\}$ for the initial point $x_0$ of \cref{alg:RCGgeneral}.
This implies that there exists $L_g > 0$ such that $\|\grad f(x_k)\|_{x_k} \le L_g$ if $t_k$ in \cref{alg:RCGgeneral} satisfies the Armijo condition~\eqref{eq:M_armijo} because~\eqref{eq:M_armijo} guarantees that $\{f(x_k)\}$ is monotonically nonincreasing.
\end{assumption}

\begin{remark}
Condition~\eqref{eq:Lip} is weaker than the condition that $\grad (f \circ R_x)$ is Lipschitz continuous for all $x \in \M$ with the same Lipschitz constant, i.e., there exists $L > 0$ such that
\begin{equation}
\label{eq:Lip_fR}
    \|\grad (f \circ R_x)(\xi) - \grad (f \circ R_x)(\eta)\|_x \le L \|\xi - \eta\|_x, \qquad \xi, \eta \in T_x \M
\end{equation}
for all $x \in \M$.
Indeed, if \eqref{eq:Lip_fR} holds for all $x \in \M$, we obtain, for any $\eta \in T_x \M$ with $\|\eta\|_x = 1$ and any $t \ge 0$,
\begin{align*}
    |\D (f \circ R_x)(t\eta)[\eta] - \D (f \circ R_x)(0)[\eta]|
    & = |\langle \grad (f \circ R_x)(t\eta) - \grad f(f \circ R_x)(0), \eta\rangle_x|\notag\\
    & \leq \|\grad (f \circ R_x)(t\eta) - \grad f(f \circ R_x)(0)\|_x\|\eta\|_x\notag\\
    & \leq L\|t\eta\|_x \|\eta\|_x = Lt.
\end{align*}
Thus, \eqref{eq:Lip} holds.

Condition~\eqref{eq:Lip} is also closely related to the condition that $f \circ R$ is radially Lipschitz continuously differentiable~\cite{AbsMahSep2008},
i.e., there exist real values $L > 0$ and $\delta > 0$ such that, for all $x \in \M$, $\eta \in T_x \M$ with $\|\eta\|_x = 1$, and $t < \delta$, it holds that
\begin{equation*}
\left|\frac{d}{d\tau}(f \circ R_x)(\tau \eta)|_{\tau = t} - \frac{d}{d\tau} (f \circ R_x)(\tau \eta)|_{\tau = 0}\right| \le L t.
\end{equation*}
Indeed, when $\delta = \infty$, this condition is equivalent to \eqref{eq:Lip}.
\end{remark}

\subsection{Riemannian version of Zoutendijk's theorem with \boldmath$\T^{(k)}$}
In Euclidean optimization, Zoutendijk's theorem plays an important role in analyzing various optimization algorithms (see, e.g., \cite{nocedal2006numerical}).
Its Riemannian version is also discussed in~\cite{ring2012optimization,sato2021riemannian,sato2015new}.
They normally state a property for a sequence generated with step lengths that satisfy the Wolfe conditions.
Here, to analyze the proposed R-CG methods with $\T^{(k)}$, we provide a similar theorem about the sequences generated by step lengths $t_k > 0$, each of which satisfies the $\T^{(k)}$-Wolfe conditions.
Note that the following result is not limited to the case of CG methods.

\begin{theorem}
\label{thm:zouten}
Consider \cref{prob:gen} on a Riemannian manifold $\M$ with a retraction $R$ and suppose \cref{assump:RCG}.
We also assume that $\T^{(k)}$ satisfies \cref{assump:Tk} and $t_k$ satisfies the $\mathscr{T}^{(k)}$-Wolfe conditions~\eqref{eq:M_armijo} and~\eqref{eq:M_Twolfe} for all $k \ge 0$.
If $\langle\grad f(x_k), \eta_k\rangle_{x_k} < 0$ for all $k \ge 0$ and there exists $\mu > 0$ such that $\|\grad f(x_k)\|_{x_k} \le \mu \|\eta_k\|_{x_k}$ for all $k \in K_2$, then we have
\begin{equation}
\label{eq:zouten}
\sum_{k=0}^\infty \frac{\langle \grad f(x_k), \eta_k\rangle_{x_k}^2}{\|\eta_k\|_{x_k}^2} < \infty,
\end{equation}
where $K_2$ is the subset of $\mathbb{N}$ in \cref{assump:Tk}.
\end{theorem}

\begin{proof}
The proof is completed by combining the discussions in~\cite{sato2021riemannian,zhu2020riemannian}.
For simplicity, let $g_k$ denote $\grad f(x_k)$ in this proof.

From the assumption and the triangle and Cauchy--Schwarz inequalities, we obtain
\begin{align}
    &  \quad \ (c_2 - 1)\langle g_k, \eta_k\rangle_{x_k} \\
    & \overset{\eqref{eq:M_Twolfe}}{\leq} \langle g_{k+1}, \T^{(k)}(\eta_k)\rangle_{x_{k+1}} - \langle g_k, \eta_k\rangle_{x_k} \notag\\
    & \leq |\langle g_{k+1}, \T^{(k)}(\eta_k) - \D R_{x_k}(t_k\eta_k)[\eta_k]\rangle_{x_{k+1}}| \notag\\
    & \qquad+ |\langle g_{k+1}, \D R_{x_k}(t_k\eta_k)[\eta_k]\rangle_{x_{k+1}} - \langle g_k, \eta_k\rangle_{x_k}| \notag\\
    & \leq |\langle g_{k+1}, \T^{(k)}(\eta_k) - \D R_{x_k}(t_k\eta_k)[\eta_k]\rangle_{x_{k+1}}| \notag\\
    & \qquad+ \|\eta_k\|_{x_k}\bigg|\D (f \circ R_{x_k})\bigg(t_k\|\eta_k\|_{x_k}\frac{\eta_k}{\|\eta_k\|_{x_k}}\bigg)\bigg[\frac{\eta_k}{\|\eta_k\|_{x_k}}\bigg] - \D (f \circ R_{x_k})(0)\bigg[\frac{\eta_k}{\|\eta_k\|_{x_k}}\bigg]\bigg| \notag\\
    & \overset{\eqref{eq:Lip}}{\leq} \|g_{k+1}\|_{x_{k+1}}\|\T^{(k)}(\eta_k) - \D R_{x_k}(t_k\eta_k)[\eta_k]\|_{x_{k+1}} + Lt_k\|\eta_k\|_{x_k}^2\notag\\
    & \overset{\eqref{eq:TkDR1}, \eqref{eq:TkDR2}}\leq
    \begin{cases}
    C\|g_{k+1}\|_{x_{k+1}}t_k\|\eta_k\|_{x_k}^2 + Lt_k\|\eta_k\|_{x_k}^2, \qquad \qquad \ \ \, k \in K_1,\\
    C\|g_{k+1}\|_{x_{k+1}}(t_k+t_k^2)\|\eta_k\|_{x_k}^2 + Lt_k\|\eta_k\|_{x_k}^2, \qquad k \in K_2
    \end{cases}.\notag
\end{align}
Therefore, we obtain
\begin{equation}
\label{eq:tk1}
    t_k \ge -\frac{1-c_2}{C\|g_{k+1}\|_{x_{k+1}}+L}\frac{\langle g_k, \eta_k\rangle_{x_k}}{\|\eta_k\|_{x_k}^2}
    \ge -\frac{1-c_2}{CL_g+L}\frac{\langle g_k, \eta_k\rangle_{x_k}}{\|\eta_k\|_{x_k}^2}, \qquad k \in K_1.
\end{equation}
For $k \in K_2$, we have
\begin{equation*}
 C\|g_{k+1}\|_{x_{k+1}}\|\eta_k\|_{x_k}^2t_k^2 + (L + C\|g_{k+1}\|_{x_{k+1}})\|\eta_k\|_{x_k}^2t_k + (1-c_2)\langle g_k, \eta_k\rangle_{x_k} \ge 0.
\end{equation*}
It follows from $t_k > 0$ and $(1-c_2)\langle g_k, \eta_k\rangle_{x_k} < 0$ that
\begin{equation*}
    t_k \ge -\frac{2(1-c_2)}{u_k}\frac{\langle g_k, \eta_k\rangle_{x_k}}{\|\eta_k\|_{x_k}^2}, \qquad k \in K_2,
\end{equation*}
where
\begin{align*}
    u_k & := L + C\|g_{k+1}\|_{x_{k+1}} + \sqrt{(L + C\|g_{k+1}\|_{x_{k+1}})^2 - 4C(1-c_2)\|g_{k+1}\|_{x_{k+1}}\frac{\langle g_k, \eta_k\rangle_{x_k}} {\|\eta_k\|_{x_k}^2}} \\
    & \leq L + C L_g + \sqrt{(L+CL_g)^2 + 4C(1-c_2)L_g \mu} =: u > 0.\notag
\end{align*}
Here, we used $\|g_k\|_{x_k} \le L_g$ and $-\langle g_k, \eta_k\rangle_{x_k} \le \|g_k\|_{x_k}\|\eta_k\|_{x_k} \le \mu \|\eta_k\|_{x_k}^2$ from the assumption $\|g_k\|_{x_k} \le \mu \|\eta_k\|_{x_k}$ for $k \in K_2$.
Using the constant $u > 0$, we obtain
\begin{equation}
\label{eq:tk2}
    t_k \geq -\frac{2(1-c_2)}{u}\frac{\langle g_k, \eta_k\rangle_{x_k}}{\|\eta_k\|_{x_k}^2}, \qquad k \in K_2.
\end{equation}
With the constant $U := \min\{(1-c_2)/(CL_g+L),2(1-c_2)/u\} > 0$, \eqref{eq:tk1} and \eqref{eq:tk2} yield $t_k \geq -U\langle g_k, \eta_k\rangle_{x_k}/\|\eta_k\|_{x_k}^2$
for all $k \ge 0$.
This and~\eqref{eq:M_armijo} yield
\begin{equation*}
f(x_{k+1})
\le f(x_k) - c_1 U\frac{\langle g_k, \eta_k\rangle_{x_k}^2}{\|\eta_k\|_{x_k}^2}
\le f(x_0) - c_1 U\sum_{j=0}^k \frac{\langle g_j, \eta_j\rangle_{x_j}^2}{\|\eta_j\|_{x_j}^2}.
\end{equation*}
It follows from \cref{assump:RCG} that $f(x) \ge f_*$ for all $x \in \M$.
Therefore, we have
\begin{equation*}
    \sum_{j=0}^k \frac{\langle g_j, \eta_j\rangle_{x_j}^2}{\|\eta_j\|_{x_j}^2} \le \frac{f(x_0) - f(x_{k+1})}{c_1 U} \leq \frac{f(x_0) - f_*}{c_1 U}.
\end{equation*}
Since the right-hand side is constant, taking the limit $k \to \infty$, we obtain the desired result.
This completes the proof.
\end{proof}

\begin{remark}
The assumption in \cref{thm:zouten} that there exists $\mu > 0$ such that $\| \grad f(x_k) \|_{x_k} \le \mu \|\eta_k\|_{x_k}$ for all $k \in K_2$ does not appear when we discuss Zoutendijk's theorem with the standard Wolfe conditions~\eqref{eq:M_armijo}--\eqref{eq:M_wolfe}~\cite{ring2012optimization,sato2021riemannian}, i.e., the case of $\T^{(k)}(\eta_k) = \mathcal{T}^R_{t_k\eta_k}(\eta_k)$, since regarding this case, we can take $K_1 = \mathbb{N}$ and $K_2 = \emptyset$.
Thus, this assumption may seem strict.
However, fortunately, in the subsequent analyses of our R-CG methods, we realize that this assumption automatically holds (requiring the generalized $\T^{(k)}$-Wolfe conditions with $c_3 = 0$ for the CD-type of R-CG methods in \cref{subsubsec:CD}).
Therefore, we do not need to explicitly suppose this assumption in the R-CG methods, and \cref{thm:zouten} is still a powerful tool, even for a general map $\T^{(k)}$.

Practically, we consider taking $K_1$ as large as possible in \cref{assump:Tk}. Therefore, $K_2$ is expected to consist of only $k$ for which~\eqref{eq:TkDR1} does not (or is not shown to) hold. Taking $K_2$ as small as possible, we require the condition $\|\grad f(x_k)\|_{x_k} \le \mu\|\eta_k\|_{x_k}$ for as the small number of $k$ as possible in \cref{thm:zouten}.
\end{remark}

\section{Convergence analyses of the R-CG methods}
\label{sec:5}
Throughout this section, we denote $g_k := \grad f(x_k) \in T_{x_k}\M$ in \cref{alg:RCGgeneral} and use the notation $\grad f(x_k)$ and $g_k$ interchangeably.
We use $\grad f(x_k)$ in the statement of the propositions and theorems and $g_k$ in their proofs and discussions for simplicity.

Considering the theoretical convergence properties, the quantity in the numerator of the formulas for $\beta_{k+1}$ is influential.
In what follows, we divide the six types of $\beta_{k+1}$ into two categories, depending on their numerators.
One consists of $\beta_{k+1}^{\RFR}$, $\beta_{k+1}^{\RDY}$, and $\beta_{k+1}^{\RCD}$ with the numerator $\|{g_{k+1}}\|_{x_{k+1}}^2$, and the other consists of $\beta_{k+1}^{\RPRP}$, $\beta_{k+1}^{\RHS}$, and $\beta_{k+1}^{\RLS}$ with the numerator $\|g_{k+1}\|_{x_{k+1}}^2 - \langle g_{k+1}, s_k \mathscr{S}^{(k)}(g_k)\rangle_{x_{k+1}}$.

\subsection{Global convergence analyses of the R-CG methods with \boldmath$\beta^{\RFR}$, $\beta^{\RDY}$, and $\beta^{\RCD}$}
\label{subsec:FR_CD_DY}

We prove the global convergence properties of the R-CG methods with $\beta^{\RFR}_{k+1}$, $\beta^{\RDY}_{k+1}$, and $\beta^{\RCD}_{k+1}$, defined in~\eqref{eq:RFR}--\eqref{eq:RCD}, and with some other related $\beta_{k+1}$.
In the subsequent analyses, a key property is that the algorithms with appropriately chosen step lengths satisfy the sufficient descent condition, i.e., there exists a constant $c > 0$ such that $\langle g_k, \eta_k\rangle_{x_k} \le -c\|g_k\|_{x_k}^2$ holds.

\subsubsection{R-CG methods with
$\beta^{\RFR}$ and its variant}
\label{subsubsec:FR}

We recall that $\beta_{k+1}^{\RFR}$ is defined in~\eqref{eq:RFR} as $\beta_{k+1}^{\text{\rm{\RFR}}} := \|g_{k+1}\|_{x_{k+1}}^2/\|g_k\|_{x_k}^2$.
The convergence analysis of the R-CG method with $\beta_{k+1} = \beta_{k+1}^{\RFR}$ can be completed following the standard discussion in the existing ones (see, e.g., \cite{sato2021riemannian}).
Here, we provide an analysis for the more general class of $\beta_{k+1}$, i.e., $\beta_{k+1}$ satisfying $|\beta_{k+1}| \le \beta_{k+1}^{\RFR}$.
We first show that such $\beta_{k+1}$ guarantees sufficient descent directions and that the ratio $\|g_k\|_{x_k} / \|\eta_k\|_{x_k}$ is bounded above.

\begin{proposition}
\label{lem:RFR}
Let sequence $\{x_k\}$ be generated by \cref{alg:RCGgeneral} with $\beta_{k+1}$ satisfying $|\beta_{k+1}| \leq \beta_{k+1}^{\text{\rm{\RFR}}}$, where $\beta_{k+1}^{\text{\rm{\RFR}}}$ is defined as~\eqref{eq:RFR}.
If, for all $k \ge 0$, $\grad f(x_k) \neq 0$ and step lengths $t_k$ satisfy the strong $\mathscr{T}^{(k)}$-Wolfe conditions~\eqref{eq:M_armijo} and~\eqref{eq:M_Tswolfe} with $0 < c_1 < c_2 < 1/2$, then we have
\begin{equation}
\label{eq:proposition_FR}
-\frac{1}{1-c_2} \le \frac{\langle \grad f(x_k), \eta_k\rangle_{x_k}}{\|\grad f(x_k)\|_{x_k}^2} \le -\frac{1-2c_2}{1-c_2}
\end{equation}
and
\begin{equation}
    \label{eq:proposition_FR2}
    \|\grad f(x_k)\|_{x_k} \le \frac{1-c_2}{1-2c_2}\|\eta_k\|_{x_k}.
\end{equation}
\end{proposition}

\begin{proof}
We define $a_k := \langle g_k, \eta_k\rangle_{x_k} / \|g_k\|_{x_k}^2$, which is the quantity in the middle of~\eqref{eq:proposition_FR}.
The proof of~\eqref{eq:proposition_FR} is completed by induction.
For $k=0$, \eqref{eq:proposition_FR} clearly holds since $a_0 = -1$ from $\eta_0 = -g_0$.
Subsequently, assume that~\eqref{eq:proposition_FR} is true for some $k \ge 0$.
Then, it follows from~\eqref{eq:RCG_direction} that
\begin{equation*}
    a_{k+1} = \frac{\langle g_{k+1}, -g_{k+1} + \beta_{k+1} s_k \T^{(k)}(\eta_k)\rangle_{x_{k+1}}}{\|g_{k+1}\|_{x_{k+1}}^2}
    = -1 + \beta_{k+1}s_k\frac{\langle g_{k+1}, \T^{(k)}(\eta_k)\rangle_{x_{k+1}}}{\|g_{k+1}\|_{x_{k+1}}^2}.
\end{equation*}
Here, \eqref{eq:M_Tswolfe} implies $|\langle g_{k+1}, \T^{(k)}(\eta_k)\rangle_{x_{k+1}}| \le - c_2 \langle g_k, \eta_k\rangle_{x_k}$.
Considering $0 < s_k \le 1$ from~\eqref{eq:sk} and $|\beta_{k+1}| \le \beta_{k+1}^{\RFR}$, we obtain
\begin{equation*}
    |a_{k+1} + 1| \leq
    \frac{\beta_{k+1}^{\RFR}}{{\|g_{k+1}\|_{x_{k+1}}^2}} |\langle g_{k+1}, \T^{(k)}(\eta_k)\rangle_{x_{k+1}}|
    \leq -c_2 \frac{\langle g_k, \eta_k\rangle_{x_{k}}}{\|g_{k}\|_{x_{k}}^2} = -c_2 a_k,
\end{equation*}
indicating $-1 + c_2 a_k \le a_{k+1} \le -1 - c_2 a_k$.
Since~\eqref{eq:proposition_FR} yields $a_k \ge -1/(1-c_2)$, we have $-1/(1-c_2) \le a_{k+1} \le -(1-2c_2)/(1-c_2)$, i.e., \eqref{eq:proposition_FR} also holds if $k$ is replaced with $k+1$.
This ends the proof of~\eqref{eq:proposition_FR} for all $k \ge 0$.

It follows from the Cauchy--Schwarz inequality $\langle g_k, \eta_k\rangle_{x_k} \ge -\|g_k\|_{x_k} \|\eta_k\|_{x_k}$ that $a_k \geq -\|\eta_k\|_{x_k}\ / \|g_k\|_{x_k}$,
which, together with~\eqref{eq:proposition_FR}, yields~\eqref{eq:proposition_FR2}.
\end{proof}

Using this theorem, we show the global convergence property of R-CG methods with $\beta_{k+1}$ satisfying $\beta_{k+1} \le |\beta^{\RFR}_{k+1}|$.
\begin{theorem}
\label{thm:FR}
Under \cref{assump:RCG}, let sequence $\{x_k\}$ be generated by \cref{alg:RCGgeneral} with $\beta_{k+1}$ satisfying $|\beta_{k+1}| \le \beta_{k+1}^{\text{\rm{\RFR}}}$, where $\beta_{k+1}^{\text{\rm{\RFR}}}$ is defined as~\eqref{eq:RFR}.
If $\T^{(k)}$ satisfies \cref{assump:Tk} and the step lengths satisfy the strong $\mathscr{T}^{(k)}$-Wolfe conditions~\eqref{eq:M_armijo} and~\eqref{eq:M_Tswolfe} with $0 < c_1 < c_2 < 1/2$, then we have
\begin{equation}
\label{eq:lim_RFR}
\liminf_{k\to\infty}\|\grad f(x_k)\|_{x_k}=0.
\end{equation}
\end{theorem}

\begin{proof}
If $g_{k_0} = 0$ holds for some $k_0 \geq 0$, then~\eqref{eq:RCG_direction} and~\eqref{eq:RFR} imply that $g_{k} = 0$ for all $k \ge k_0$; thus, \eqref{eq:lim_RFR} holds.

Subsequently, we assume $g_k \neq 0$ for all $k \geq 0$ and prove~\eqref{eq:lim_RFR} by contradiction.
To this end, we assume that~\eqref{eq:lim_RFR} does not hold, indicating that there exists $\varepsilon > 0$ such that $\|g_k\|_{x_k} \ge \varepsilon$ for all $k \ge 0$.
Furthermore, since the assumption in \cref{lem:RFR} holds, we have~\eqref{eq:proposition_FR} and \eqref{eq:proposition_FR2}.
Hence, the assumption in \cref{thm:zouten} is also ensured to imply~\eqref{eq:zouten}.
With $c := (1+c_2) / (1-c_2)$, we can evaluate $\|\eta_{k+1}\|_{x_{k+1}}^2$ as
\begin{align}
    \|\eta_{k+1}\|^2_{x_{k+1}} &\overset{\eqref{eq:RCG_direction}}{=} \|g_{k+1}\|_{x_{k+1}}^2 - 2 \beta_{k+1} s_k \langle g_{k+1}, \T^{(k)}(\eta_k)\rangle_{x_{k+1}} + \beta_{k+1}^2 s_k^2 \|\T^{(k)}(\eta_k)\|_{x_{k+1}}^2\notag\\
    & \overset{\eqref{eq:sk}}{\le} \|g_{k+1}\|_{x_{k+1}}^2 + 2|\beta_{k+1}| |\langle g_{k+1}, \T^{(k)}(\eta_k)\rangle_{x_{k+1}}| + \beta_{k+1}^2 \|\eta_k\|_{x_k}^2\notag\\
    &\overset{\eqref{eq:M_Tswolfe}}{\le} \|g_{k+1}\|_{x_{k+1}}^2 - 2c_2\beta_{k+1}^{\RFR}\langle g_k, \eta_k\rangle_{x_k} + \beta_{k+1}^2\|\eta_k\|_{x_k}^2\notag\\
    &\overset{\eqref{eq:proposition_FR}}{\le} \|g_{k+1}\|_{x_{k+1}}^2 + \frac{2c_2}{1-c_2}\beta^{\RFR}_{k+1}\|g_k\|_{x_k}^2 + (\beta^{\RFR}_{k+1})^2\|\eta_k\|_{x_k}^2\notag\\
    &\overset{\eqref{eq:RFR}}{=} c\|g_{k+1}\|_{x_{k+1}}^2 + (\beta^{\RFR}_{k+1})^2\|\eta_k\|_{x_k}^2.\notag
\end{align}
This recurrence relation together with $\|\eta_0\|_{x_0} =\|g_0\|_{x_0}$ and $c > 1$ gives
\begin{align}
    \|\eta_k\|_{x_k}^2 &\le c\bigg(\|g_k\|_{x_k}^2 + \sum_{j = 1}^{k-1} (\beta^{\RFR}_k)^2 \cdots (\beta^{\RFR}_{j+1})^2\|g_j\|_{x_j}^2\bigg) + (\beta^{\RFR}_k)^2 \cdots (\beta^{\RFR}_1)^2\|\eta_0\|_{x_0}^2\notag\\
    &< c \|g_k\|_{x_k}^4\sum_{j=0}^k \|g_{j}\|_{x_{j}}^{-2}
    \le \frac{c}{\varepsilon^{2}}\|g_k\|_{x_k}^4 (k+1).\notag
\end{align}
Therefore, using~\eqref{eq:proposition_FR} again, we obtain
\begin{equation*}
    \sum_{j = 0}^k \frac{\langle g_j, \eta_j\rangle_{x_j}^2}{\|\eta_j\|_{x_j}^2}
    > \frac{\varepsilon^2}{c}\sum_{j=0}^k\frac{\langle g_j, \eta_j\rangle_{x_j}^2}{\|g_j\|_{x_j}^4} \frac{1}{j+1}
    \geq \frac{\varepsilon^2(1-2c_2)^2}{c(1-c_2)^2}\sum_{j=1}^{k+1}\frac{1}{j}.
\end{equation*}
Taking the limit $k \to \infty$, the right-hand side, and hence left-hand side, diverge to $\infty$.
This contradicts~\eqref{eq:zouten}, completing the proof.
\end{proof}

The global convergence property is ensured for \cref{alg:RCGgeneral} with $\beta_{k+1} = \beta^{\RFR}_{k+1}$ as a corollary of \cref{thm:FR}.
Because we have analyzed a class of $\beta_{k+1}$, rather than the specific $\beta^{\RFR}_{k+1}$ only, we can apply the results here to other R-CG methods such as one with $\beta^{\RCD}_{k+1}$ (see \cref{subsubsec:CD}).

\subsubsection{R-CG methods with $\beta^{\RDY}$ and its variant}
In this subsection, we give a global convergence analysis of the R-CG methods with a class of $\beta_{k+1}$ containing $\beta^{\RDY}_{k+1} := \|g_{k+1}\|_{x_{k+1}}^2 / (\langle g_{k+1}, s_k\T^{(k)}(\eta_k)\rangle_{x_{k+1}} - \langle g_k, \eta_k\rangle_{x_k})$ in~\eqref{eq:RDY}.
We first show that such R-CG methods generate descent search directions and that, with additional assumptions, the search directions are sufficient descent directions.
Thereafter, we build a global convergence result.
\cref{lem:RDY} and \cref{thm:DY} are inspired by, but more general than, the results in~\cite{sato2021riemannian,zhu2020riemannian}.
Therefore, their proofs are not verbatim.

\begin{proposition}
\label{lem:RDY}
Let sequence $\{x_k\}$ be generated by \cref{alg:RCGgeneral} with $\beta_{k+1}$ satisfying $0 \le \beta_{k+1} \leq \beta_{k+1}^{\text{\rm{\RDY}}}$, where $\beta_{k+1}^{\text{\rm{\RDY}}}$ is defined as~\eqref{eq:RDY}.
Assume that, for all $k \ge 0$, $\grad f(x_k) \neq 0$ and the step lengths $t_k$ satisfy the $\mathscr{T}^{(k)}$-Wolfe conditions~\eqref{eq:M_armijo} and~\eqref{eq:M_Twolfe} with $0 < c_1 < c_2 < 1$.
Then, the algorithm is well-defined, i.e., for all $k \ge 0$, $\beta_{k+1}^{\text{\rm{\RDY}}} > 0$ holds, and thus $\beta_{k+1}$ with $0 \le \beta_{k+1} \le \beta_{k+1}^{\text{\rm{\RDY}}}$ exist, and we have
\begin{equation}
\label{eq:proposition_DY1}
    \langle \grad f(x_k), \eta_k\rangle_{x_k} < \min\{0, \langle \grad f(x_{k+1}), s_k\T^{(k)}(\eta_k)\rangle_{x_{k+1}}\}.
\end{equation}
Furthermore, if, for an arbitrary $k \ge 1$\footnote{For $k=0$, \eqref{eq:proposition_DY2} and \eqref{eq:proposition_DY3} clearly hold without any assumption since $\eta_0 = -\grad f(x_0)$.}, $t_{k-1}$ satisfies the generalized $\T^{(k-1)}$-Wolfe conditions~\eqref{eq:M_armijo} and~\eqref{eq:M_Tgwolfe} with $0 < c_1 < c_2 < 1$ and $c_3 \geq 0$, then for this $k$, it holds that
\begin{equation}
\label{eq:proposition_DY2}
-\frac{1}{1-c_2} \le \frac{\langle \grad f(x_k), \eta_k\rangle_{x_k}}{\|\grad f(x_k)\|_{x_k}^2} \le -\frac{1}{1+c_3}
\end{equation}
and
\begin{equation}
    \label{eq:proposition_DY3}
    \|\grad f(x_k)\|_{x_k} \le (1+c_3)\|\eta_k\|_{x_k}.
\end{equation}
\end{proposition}

\begin{proof}
We first prove $\beta^{\RDY}_{k+1} > 0$ and~\eqref{eq:proposition_DY1} by induction.

For $k=0$, $\langle g_0, \eta_0\rangle_{x_0} = -\|g_0\|_{x_0}^2 < 0$.
If $\langle g_1, \T^{(0)}(\eta_0)\rangle_{x_0} \geq 0$, then from $s_0 > 0$, \eqref{eq:proposition_DY1} holds.
Otherwise, from $\langle g_1, \T^{(0)}(\eta_0)\rangle_{x_1}, \langle g_0, \eta_0\rangle_{x_0} < 0$, $s_0, c_2 < 1$, and~\eqref{eq:M_Twolfe}, we have $\langle g_1, s_0\T^{(0)}(\eta_0)\rangle_{x_1} > \langle g_1, \T^{(0)}(\eta_0)\rangle_{x_1} \geq c_2\langle g_0, \eta_0\rangle_{x_0} > \langle g_0, \eta_0\rangle_{x_0}$, indicating that~\eqref{eq:proposition_DY1} holds.
Moreover, \eqref{eq:proposition_DY1} directly ensures $\beta^{\RDY}_1 > 0$.

Now assume that, for some $k \ge 0$, $\beta^{\RDY}_{k+1} > 0$ and~\eqref{eq:proposition_DY1} hold.
Then, $\beta_{k+1}$ satisfying $0 \le \beta_{k+1} \le \beta^{\RDY}_{k+1}$ exists.
If $0 < \beta_{k+1} \le \beta^{\RDY}_{k+1}$, we obtain
\begin{align*}
    & \quad \langle g_{k+1}, \eta_{k+1}\rangle_{x_{k+1}} \\
    &\overset{\eqref{eq:RCG_direction}}{=} -\|g_{k+1}\|_{x_{k+1}}^2 + \beta_{k+1}\langle g_{k+1}, s_k\T^{(k)}(\eta_k)\rangle_{x_{k+1}}\\
    &= -\|g_{k+1}\|_{x_{k+1}}^2 + \beta_{k+1}(\langle g_{k+1}, s_k\T^{(k)}(\eta_k)\rangle_{x_{k+1}} - \langle g_k, \eta_k\rangle_{x_k}) + \beta_{k+1}\langle g_k, \eta_k\rangle_{x_k}\\
    &\overset{\eqref{eq:proposition_DY1}}{\le} -\|g_{k+1}\|_{x_{k+1}}^2 + \beta^{\RDY}_{k+1}(\langle g_{k+1}, s_k\T^{(k)}(\eta_k)\rangle_{x_{k+1}} - \langle g_k, \eta_k\rangle_{x_k}) + \beta_{k+1}\langle g_k, \eta_k\rangle_{x_k}\\
    &\overset{\eqref{eq:RDY}}{=} \beta_{k+1} \langle g_k, \eta_k\rangle_{x_k} < 0.
\end{align*}
If $\beta_{k+1} = 0$, then we have $\langle g_{k+1}, \eta_{k+1}\rangle_{x_{k+1}} = -\|g_{k+1}\|_{x_{k+1}}^2 < 0$.
We can also prove the inequality $\langle g_{k+1}, \eta_{k+1}\rangle_{x_{k+1}} < \langle g_{k+2}, s_{k+1}\T^{(k+1)}(\eta_{k+1})\rangle_{x_{k+2}}$ as in the previous paragraph.
Therefore, \eqref{eq:proposition_DY1} holds if $k$ is replaced with $k+1$.
Hence, \eqref{eq:proposition_DY1} is proved for all $k \ge 0$.

We proceed to prove~\eqref{eq:proposition_DY2} and~\eqref{eq:proposition_DY3} for any $k \ge 1$ with which $t_{k-1}$ satisfies the generalized $\T^{(k-1)}$-Wolfe conditions~\eqref{eq:M_armijo} and~\eqref{eq:M_Tgwolfe}.
It follows from~\eqref{eq:RCG_direction} that
\begin{equation}
\label{eq:DYproof}
    \langle g_{k}, \eta_{k}\rangle_{x_{k}} = -\|g_k\|_{x_k}^2 + \beta_k\langle g_k, s_{k-1}\T^{(k-1)}(\eta_{k-1})\rangle_{x_k}.
\end{equation}
We now prove the first inequality in~\eqref{eq:proposition_DY2}.
Here, from~\eqref{eq:M_Tgwolfe} and $s_k \le 1$, we observe that $\langle g_k, s_{k-1} \T^{(k-1)}(\eta_{k-1})\rangle_{x_k} \ge s_{k-1} c_2 \langle g_{k-1}, \eta_{k-1}\rangle_{x_{k-1}} \ge c_2 \langle g_{k-1}, \eta_{k-1}\rangle_{x_{k-1}}$ ($< 0$).
Therefore, it follows from~\eqref{eq:DYproof}, $0 \le \beta_{k} \le \beta^{\RDY}_{k}$, \eqref{eq:RDY}, and~\eqref{eq:M_Tgwolfe} that
\begin{align*}
    \langle g_k, \eta_k\rangle_{x_k}
    &\ge -\|g_k\|_{x_k}^2 + c_2 \beta_k \langle g_{k-1}, \eta_{k-1}\rangle_{x_{k-1}}\\
    &\ge -\|g_k\|_{x_k}^2 + c_2 \beta^{\RDY}_k \langle g_{k-1}, \eta_{k-1}\rangle_{x_{k-1}}\\
    &= \|g_k\|_{x_k}^2\bigg(-1 + \frac{c_2\langle g_{k-1}, \eta_{k-1}\rangle_{x_{k-1}}}{\langle g_k, s_{k-1}\T^{(k-1)}(\eta_{k-1})\rangle_{x_k} - \langle g_{k-1}, \eta_{k-1}\rangle_{x_{k-1}}}\bigg) \\
    &\ge \|g_k\|_{x_k}^2\bigg(-1 + \frac{c_2\langle g_{k-1}, \eta_{k-1}\rangle_{x_{k-1}}}{c_2\langle g_{k-1}, \eta_{k-1}\rangle_{x_{k-1}} - \langle g_{k-1}, \eta_{k-1}\rangle_{x_{k-1}}}\bigg) 
    =\frac{\|g_k\|_{x_k}^2}{c_2-1}.
\end{align*}
For the second inequality in~\eqref{eq:DYproof}, if $\langle g_k, s_{k-1}\T^{(k-1)}(\eta_{k-1})\rangle_{x_k} \le 0$, then \eqref{eq:DYproof} and $\beta_{k} \ge 0$ yield
\begin{equation*}
    \langle g_k, \eta_k\rangle_{x_k} \le -\|g_k\|_{x_k}^2 \le -\frac{1}{1+c_3}\|g_k\|_{x_k}^2.
\end{equation*}
Otherwise (i.e., if $\langle g_k, s_{k-1}\T^{(k-1)}(\eta_{k-1})\rangle_{x_k} > 0$), \eqref{eq:DYproof}, $\beta_{k} \le \beta^{\RDY}_{k}$, and~\eqref{eq:RDY} give
\begin{align*}
\langle g_k, \eta_k\rangle_{x_k} &\le -\|g_k\|_{x_k}^2 + \beta^{\RDY}_k \langle g_k, s_{k-1}\T^{(k-1)}(\eta_{k-1})\rangle_{x_k}\\
&= \frac{\|g_k\|_{x_k}^2\langle g_{k-1}, \eta_{k-1}\rangle_{x_{k-1}}}{\langle g_k, s_{k-1}\T^{(k-1)}(\eta_{k-1})\rangle_{x_k} - \langle g_{k-1}, \eta_{k-1}\rangle_{x_{k-1}}}.
\end{align*}
Noting $\|g_k\|_{x_k}^2 \langle g_{k-1}, \eta_{k-1}\rangle_{x_{k-1}} < 0$ and evaluating $\langle g_k, s_{k-1}\T^{(k-1)}(\eta_{k-1})\rangle_{x_k}$ by using~\eqref{eq:M_Tgwolfe} as $\langle g_k, s_{k-1}\T^{(k-1)}(\eta_{k-1})\rangle_{x_k} \le -c_3 \langle g_{k-1}, \eta_{k-1}\rangle_{x_{k-1}}$, we obtain
\begin{equation*}
\langle g_k, \eta_k\rangle_{x_k} < \frac{\|g_k\|_{x_k}^2}{-c_3-1},
\end{equation*}
indicating that the second inequality in~\eqref{eq:proposition_DY2} always holds.
Finally, \eqref{eq:proposition_DY3} is a direct consequence from~\eqref{eq:proposition_DY2} and the Cauchy--Schwarz inequality,
completing the proof.
\end{proof}

\begin{theorem}
\label{thm:DY}
Under \cref{assump:RCG}, let sequence $\{x_k\}$ be generated by \cref{alg:RCGgeneral} with $\T^{(k)}$ satisfying \cref{assump:Tk}.
We assume that $\beta_{k+1}$ in the algorithm satisfies $0 \le \beta_{k+1} \le \beta^{\text{\rm{\RDY}}}_{k+1}$ for all $k \ge 0$ and $t_k$ satisfies the $\T^{(k)}$-Wolfe conditions~\eqref{eq:M_armijo} and~\eqref{eq:M_Twolfe} with $0 < c_1 < c_2 < 1$.
Furthermore, assume that for $k \in K_2 - \{0\}$, $t_{k-1}$ satisfies the generalized $\T^{(k-1)}$-Wolfe conditions~\eqref{eq:M_armijo} and~\eqref{eq:M_Tgwolfe} with $0 < c_1 < c_2 < 1$ and $c_3 \geq 0$, where $\beta^{\text{\rm{\RDY}}}_{k+1}$ is defined as~\eqref{eq:RDY}, and $K_2$ is the index set in \cref{assump:Tk}.
Then, we have
\begin{equation}
\label{eq:lim_RDY}
\liminf_{k\to\infty}\|\grad f(x_k)\|_{x_k}=0.
\end{equation}
\end{theorem}

\begin{proof}
It is sufficient to show~\eqref{eq:lim_RDY} for the case $g_k \neq 0$ for all $k \ge 0$.
From~\eqref{eq:RCG_direction}, we have $\eta_{k+1} + g_{k+1} = \beta_{k+1}s_k\T^{(k)}(\eta_k)$.
Taking the norm and squaring, we obtain
\begin{equation}
\label{eq:thmRDY}
    \|\eta_{k+1}\|_{x_{k+1}}^2 = \beta_{k+1}^2 s_k^2 \|\T^{(k)}(\eta_k)\|_{x_{k+1}}^2 -2\langle g_{k+1}, \eta_{k+1}\rangle_{x_{k+1}} - \|g_{k+1}\|_{x_{k+1}}^2.
\end{equation}
In the proof of \cref{lem:RDY}, we showed $\langle g_{k+1}, \eta_{k+1}\rangle_{x_{k+1}} \le \beta_{k+1} \langle g_k, \eta_k\rangle_{x_k} \le 0$ (the two equal signs do not hold simultaneously).
Thus, $\beta_{k+1}^2 \langle g_k, \eta_k\rangle_{x_k}^2 \le \langle g_{k+1}, \eta_{k+1}\rangle_{x_{k+1}}^2$.
Dividing both sides of~\eqref{eq:thmRDY} by $\langle g_{k+1}, \eta_{k+1}\rangle_{x_{k+1}}^2 > 0$, we obtain
\begin{align}
    \frac{\|\eta_{k+1}\|_{x_{k+1}}^2}{\langle g_{k+1}, \eta_{k+1}\rangle_{x_{k+1}}^2} &\le  \frac{s_k^2\|\T^{(k)}(\eta_k)\|_{x_{k+1}}^2}{\langle g_{k}, \eta_{k}\rangle_{x_{k}}^2}- \frac{2}{\langle g_{k+1}, \eta_{k+1}\rangle_{x_{k+1}}} - \frac{\|g_{k+1}\|_{x_{k+1}}^2}{\langle g_{k+1}, \eta_{k+1}\rangle_{x_{k+1}}^2}\notag\\
    &\overset{\eqref{eq:sk}}{\le} \frac{\|\eta_k\|_{x_k}^2}{\langle g_k, \eta_k\rangle_{x_k}^2} + \frac{1}{\|g_{k+1}\|_{x_{k+1}}^2} - \bigg(\frac{1}{\|g_{k+1}\|_{x_{k+1}}} + \frac{\|g_{k+1}\|_{x_{k+1}}}{\langle g_{k+1}, \eta_{k+1}\rangle_{x_{k+1}}}\bigg)^2\notag\\
    &\le \frac{\|\eta_k\|_{x_k}^2}{\langle g_k, \eta_k\rangle_{x_k}^2} + \frac{1}{\|g_{k+1}\|_{x_{k+1}}^2}.\label{eq:thmRDY2}
\end{align}

To accomplish the proof by contradiction,
we assume $\liminf_{k\to\infty}\|g_k\|_{x_k} > 0$, which, together with $g_k \neq 0$ for all $k \geq 0$, implies that there exists $\varepsilon > 0$ such that $\|g_k\|_{x_k} \geq \varepsilon$ for all $k \ge 0$.
Therefore, from~\eqref{eq:thmRDY2}, we obtain
\begin{equation*}
    \frac{\|\eta_k\|_{x_k}^2}{\langle g_k, \eta_k\rangle_{x_k}^2} \le \frac{\|\eta_0\|_{x_0}^2}{\langle g_0, \eta_0\rangle_{x_0}^2} + \sum_{j=1}^{k}\frac{1}{\|g_j\|_{x_j}^2} = \sum_{j=0}^k \frac{1}{\|g_j\|_{x_j}^2} \le \frac{k+1}{\varepsilon^2},
\end{equation*}
which gives
\begin{equation}
\label{eq:proposition_DY4}
    \sum_{k=0}^N\frac{\langle g_k, \eta_k\rangle_{x_k}^2}{\|\eta_k\|_{x_k}^2} \geq \varepsilon^2\sum_{k=1}^{N+1}\frac{1}{k} \to \infty
\end{equation}
as $N \to \infty$.
On the other hand, \cref{lem:RDY} indicates that the assumption in \cref{thm:zouten} (Zoutendijk's theorem) holds, and we have~\eqref{eq:zouten}, contradicting~\eqref{eq:proposition_DY4}.
Therefore, we deduce that~\eqref{eq:lim_RDY} must hold, completing the proof.
\end{proof}

The DY-type of R-CG methods have an advantage over the FR-types in that they do not require the strong $\T^{(k)}$-Wolfe conditions.
Although the generalized $\T^{(k)}$-Wolfe conditions are required for $k$ such that \eqref{eq:TkDR2} does not hold, the constant $c_3 \geq 0$ can be taken as any large constant.
Therefore, the conditions are not too restrictive and much weaker than the strong $\T^{(k)}$-Wolfe conditions.

\subsubsection{R-CG methods with $\beta^{\RCD}$}
\label{subsubsec:CD}
For $\beta^{\RCD}_{k+1} := - \|g_{k+1}\|_{x_{k+1}}^2 / \langle  g_k, \eta_k\rangle_{x_k}$, defined as~\eqref{eq:RCD}, the following proposition is crucial.
\begin{proposition}
\label{prop:CD}
Let sequence $\{x_k\}$ be generated by \cref{alg:RCGgeneral} with $\beta_{k+1}$ satisfying $0 \le \beta_{k+1} \le \beta_{k+1}^{\text{\rm{\RCD}}}$, where $\beta_{k+1}^{\text{\rm{\RCD}}}$ is defined as~\eqref{eq:RCD}.
Assume that, for all $k \ge 0$, $\grad f(x_k) \neq 0$ and step lengths $t_k$ satisfy the $\T^{(k)}$-generalized Wolfe conditions~\eqref{eq:M_armijo} and~\eqref{eq:M_Tgwolfe} with $c_3 = 0$.
Then, the algorithm is well-defined, i.e., for all $k \ge 0$, $\beta_{k+1}^{\text{\rm{\RCD}}} > 0$; thus, $\beta_{k+1}$ satisfying $0 \le \beta_{k+1} \le \beta_{k+1}^{\text{\rm{\RCD}}}$ exists.
Furthermore, we have the sufficient descent condition on $\eta_k$ as $\langle \grad f(x_k), \eta_k\rangle_{x_k} \le -\|\grad f(x_k)\|_{x_k}^2$ and $0 \le \beta_{k+1} \le \beta^{\text{\rm{\RFR}}}_{k+1}$ for all $k \ge 0$, where $\beta^{\text{\rm{\RFR}}}_{k+1}$ is defined as~\eqref{eq:RFR}.
In particular, $0 < \beta^{\text{\rm{\RCD}}}_{k+1} \le \beta^{\text{\rm{\RFR}}}_{k+1}$ holds.
\end{proposition}

\begin{proof}
We first show $\langle g_k, \eta_k\rangle_{x_k} \le -\|g_k\|_{x_k}^2$ by induction.
If this inequality holds, then $\beta^{\RCD}_{k+1} > 0$ directly follows from $\langle g_k, \eta_k\rangle_{x_k} < 0$, and $\beta_{k+1} \in [0, \beta^{\RCD}_{k+1}]$ actually exists.
For $k=0$, it is clear that $\langle g_0, \eta_0\rangle_{x_0} = -\|g_0\|_{x_0}^2$.
Subsequently, we assume that $\langle g_k, \eta_k\rangle_{x_k} \le -\|g_k\|_{x_k}^2$ ($< 0$) for some $k \ge 0$.
Then, considering the inequality $\langle g_{k+1}, \T^{(k)}(\eta_k)\rangle_{x_{k+1}} \le 0$ from~\eqref{eq:M_Tgwolfe} with $c_3 = 0$, we use~\eqref{eq:RCG_direction} and~\eqref{eq:RCD} to obtain
\begin{align*}
    \frac{\langle g_{k+1}, \eta_{k+1}\rangle_{x_{k+1}}}{\|g_{k+1}\|_{x_{k+1}}^2}
    &= \frac{\langle g_{k+1}, -g_{k+1} + \beta_{k+1}s_k\T^{(k)}(\eta_k)\rangle_{x_{k+1}}}{\|g_{k+1}\|_{x_{k+1}}^2}\\
    &= -1 - s_k\frac{\beta_{k+1}}{\beta^{\RCD}_{k+1}}\frac{\langle g_{k+1}, \T^{(k)}(\eta_k)\rangle_{x_{k+1}}}{\langle g_{k}, \eta_k\rangle_{x_{k}}} \le -1.
\end{align*}
Thus, $\langle g_{k+1}, \eta_{k+1}\rangle_{x_{k+1}} \le -\|g_{k+1}\|_{x_{k+1}}^2$ holds as desired, and by induction, we have $\langle g_k, \eta_k\rangle_{x_k} \le -\|g_k\|_{x_k}^2$ for all $k \ge 0$.

Furthermore, this result directly leads to
\begin{equation*}
    \beta^{\RCD}_{k+1} = \frac{\|g_{k+1}\|_{x_{k+1}}^2}{-\langle g_k, \eta_k\rangle_{x_k}} \le \frac{\|g_{k+1}\|_{x_{k+1}}^2}{\|g_k\|_{x_k}^2} = \beta^{\RFR}_{k+1}.
\end{equation*}
Combining this relationship with $0 \le \beta_{k+1} \le \beta^{\RCD}_{k+1}$ yields $0 \le \beta_{k+1} \le \beta^{\RFR}_{k+1}$.
\end{proof}

As a special case of the result in \cref{prop:CD}, we have $|\beta^{\RCD}_{k+1}| \le \beta^{\RFR}_{k+1}$ when we choose $\beta_{k+1} \equiv \beta^{\RCD}_{k+1}$ for all $k \ge 0$.
Therefore, as a corollary of \cref{thm:FR}, we obtain the convergence result for $\beta^{\RCD}_{k+1}$, requiring $t_k$ to satisfy the generalized $\T^{(k)}$-Wolfe conditions with $0<c_1<c_2<1/2$ and $c_3 = 0$.
Furthermore, in fact, as the following theorem sates, the condition on $c_2$ can be weaken as $0 < c_1 < c_2 < 1$, and $\beta_{k+1}$ can be any value satisfying $0 \le \beta_{k+1} \le \beta^{\RCD}_{k+1}$.

\begin{theorem}
\label{thm:CD}
Under \cref{assump:RCG}, let sequence $\{x_k\}$ be generated by \cref{alg:RCGgeneral} with $\beta_{k+1}$ satisfying $0 \le \beta_{k+1} \le \beta^{\text{\rm{\RCD}}}_{k+1}$, where $\beta_{k+1}^{\text{\rm{\RCD}}}$ is defined as~\eqref{eq:RCD}.
If $\T^{(k)}$ satisfies \cref{assump:Tk} and the step lengths satisfy the generalized $\mathscr{T}^{(k)}$-Wolfe conditions~\eqref{eq:M_armijo} and~\eqref{eq:M_Tgwolfe} with $0 < c_1 < c_2 < 1$ and $c_3 = 0$, then we have
\begin{equation*}
\label{eq:lim_RCD}
\liminf_{k\to\infty}\|\grad f(x_k)\|_{x_k}=0.
\end{equation*}
\end{theorem}

\begin{proof}
Following the discussion in the proof of \cref{lem:RFR}, we can prove $\displaystyle-\frac{1}{1-c_2} \le \frac{\langle g_k, \eta_k\rangle_{x_k}}{\|g_k\|_{x_k}^2} \le -1$
and
$\|g_k\|_{x_k} \le \|\eta_k\|_{x_k}$ instead of~\eqref{eq:proposition_FR} and~\eqref{eq:proposition_FR2}, respectively.
Therefore, the assumption in Zoutendijk's theorem (\cref{thm:zouten}) is satisfied, and the subsequent proof is completely the same as that of \cref{thm:FR}.
\end{proof}

The discussion on FR- and DY-types of R-CG methods here is partly similar to that in previous studies, where some specific choices of $s_k$ and $\T^{(k)}$ are utilized.
However, the analyses provided in this section are meaningful and not trivial since they address our general framework of R-CG methods (i.e., \cref{alg:RCGgeneral}) and more general classes of $\beta_{k+1}$.
Furthermore, to the author's knowledge, no discussion on the CD-type of R-CG method has been conducted, even for specific $\T^{(k)}$ such as parallel translation or vector transport.

\subsection{Global convergence analyses of R-CG methods with variants of \boldmath$\beta^{\RPRP}$, $\beta^{\RHS}$, and $\beta^{\RLS}$}
\label{subsec:PRP_HS_LS}

While some existing studies discuss the FR- and DY-types of R-CG methods, the theoretical properties of the PRP-, HS-, and LS-types of R-CG methods are not well known until now.
Furthermore, even in Euclidean spaces, these three types of CG methods with the (strong) Wolfe conditions are not generally guaranteed to converge.
Therefore, various variants are proposed and analyzed.
For example, $\beta_{k+1}^{\PRP+} := \max\{\beta_{k+1}^{\PRP}, 0\}$ is known to generate convergent sequences under some assumptions in the Euclidean case~\cite{gilbert1992global}.
Some comprehensive surveys on the Euclidean CG methods are found in~\cite{andrei2008conjugate,andrei2020nonlinear,hager2006survey,narushima2014survey}.
In this subsection, we generalize some examples of such variants by exploiting the theoretical results in \cref{subsec:FR_CD_DY}.

Considering the Euclidean CG methods, an important feature of the PRP-, \mbox{HS-,} and LS-types is that they can avoid jamming, which may occur in the FR-, DY-, and CD-types.
This is because the quantity in the numerator of $\beta^{\PRP}_{k+1}$, $\beta^{\HS}_{k+1}$, and $\beta^{\LS}_{k+1}$ becomes close to $0$ when $x_{k+1} \approx x_k$, and the search direction is almost the steepest descent direction $-\nabla f(x_{k+1})$.
This phenomenon can also be explained in the R-CG methods, assuming that $l_k \S^{(k)}(g_k) \approx g_k$ when $x_{k+1} \approx x_k$.
For example, consider $\beta^{\RPRP}_{k+1} := (\|g_{k+1}\|_{x_{k+1}}^2 - \langle g_{k+1}, l_k \S^{(k)}(g_k)\rangle_{x_{k+1}}) / \|g_k\|_{x_k}^2$ as defined in~\eqref{eq:RPRP}.
If $x_{k+1} \approx x_k$, then $g_{k+1} \approx g_k$ from the continuity of $\grad f$, and the numerator is approximated as $\|g_{k+1}\|_{x_{k+1}}^2 - \langle g_{k+1}, l_k \S^{(k)}(g_k)\rangle_{x_{k+1}} \approx \|g_k\|_{x_k}^2 - \langle g_k, g_k\rangle_{x_k} = 0$.
Therefore, the subsequent search direction in \cref{alg:RCGgeneral} is $\eta_{k+1} \approx -g_{k+1}$, which is the negative gradient of $f$ at $x_{k+1}$.
Therefore, the PRP-type of R-CG methods are equipped with an automatic restart strategy, indicating that the search direction is almost reset as the steepest descent direction when $x_{k+1} \approx x_k$.
The same is also applied to the HS- and LS-types of R-CG methods.

As mentioned, although the PRP-, HS-, and LS-types of (R-)CG methods may be practically superior to the FR-, DY-, and CD-types, they do not necessarily generate convergent sequences.
Here, we observe that $\beta^{\FR}_{k+1}$ and $\beta^{\PRP}_{k+1}$ have the same form of denominator.
Therefore, the PRP-type R-CG methods can be regarded as practically modified versions of the FR-type R-CG methods so that they have the aforementioned restart mechanism, while the FR-types have theoretically better convergence properties than the PRP-types.
Based on this discussion, a natural modification of $\beta^{\RPRP}_{k+1}$ is $\beta_{k+1} = \max\{0, \min\{\beta^{\RPRP}_{k+1}, \beta^{\RFR}_{k+1}\}\}$,
which ensures the condition $0 \le \beta_{k+1} \le \beta^{\RFR}_{k+1}$ in \cref{thm:FR}.
We can develop this discussion for HS--DY- and CD--LS-types.

\subsubsection{R-CG methods with modified \boldmath$\beta^{\RPRP}$}
\label{subsubsec:PRP}
Based on the above discussion, a practical implementation of the PRP-type $\beta_{k+1}$ defined in~\eqref{eq:RPRP}, which is $\beta^{\RPRP}_{k+1} := (\|g_{k+1}\|_{x_{k+1}}^2 - \langle g_{k+1}, l_k \S^{(k)}(g_k)\rangle_{x_{k+1}}) / \|g_k\|_{x_k}^2$, may be to combine it with $\beta^{\RFR}_{k+1}$ as $\beta_{k+1} = \max\{0, \min\{\beta^{\RPRP}_{k+1}, \beta^{\RFR}_{k+1}\}\}$.
The Euclidean version of this approach (i.e., $\beta_{k+1} = \max\{0, \min\{\beta^{\PRP}_{k+1}, \beta^{\FR}_{k+1}\}\}$) is mentioned in~\cite{hu1991global}.
Because $0 \le \beta_{k+1} \le \beta^{\RFR}_{k+1}$ holds, \cref{thm:FR} implies the following result.

\begin{theorem}
\label{thm:PRP}
Let $\{x_k\}$ be a sequence generated by \cref{alg:RCGgeneral} under the assumption in \cref{thm:FR} and $\beta_{k+1} = \beta_{k+1}^{\text{\rm{\RPRPFR}}} := \max\{0, \min\{\beta^{\text{\rm{\RPRP}}}_{k+1}, \beta^{\text{\rm{\RFR}}}_{k+1}\}\}$.
Then, we have $\liminf_{k \to \infty}\|\grad f(x_k)\|_{x_k} = 0$.
\end{theorem}

\subsubsection{R-CG methods with modified \boldmath$\beta^{\RHS}$}
\label{subsubsec:HS}
In~\eqref{eq:RHS}, $\beta^{\RHS}_{k+1} := (\|g_{k+1}\|_{x_{k+1}}^2 - \langle g_{k+1}, l_k\S^{(k)}(g_k)\rangle_{x_{k+1}}) / (\langle g_{k+1}, s_k\T^{(k)}(\eta_k)\rangle_{x_{k+1}} - \langle g_k, \eta_k\rangle_{x_k})$ is proposed as the HS-type.
We develop the idea discussed in \cref{subsubsec:PRP} and use \cref{thm:DY} to propose and analyze the following algorithm.
Its Euclidean counterpart $\beta_{k+1} = \max\{0,\min\{\beta^{\HS}_{k+1},\beta^{\DY}_{k+1}\}\}$ is proposed in~\cite{dai2001efficient}.

\begin{theorem}
\label{thm:HS}
Let $\{x_k\}$ be a sequence generated by \cref{alg:RCGgeneral} under the assumption in \cref{thm:DY} and $\beta_{k+1} = \beta_{k+1}^{\text{\rm{\RHSDY}}} := \max\{0, \min\{\beta^{\text{\rm{\RHS}}}_{k+1}, \beta^{\text{\rm{\RDY}}}_{k+1}\}\}$.
Then, we have $\liminf_{k \to \infty}\|\grad f(x_k)\|_{x_k} = 0$.
\end{theorem}

\subsubsection{R-CG methods with modified \boldmath$\beta^{\RLS}$}
\label{subsubsec:LS}
Finally, we discuss the LS-type $\beta^{\RLS}_{k+1} := (\|g_{k+1}\|_{x_{k+1}}^2 - \langle g_{k+1}, l_k \S^{(k)}(g_k)\rangle_{x_{k+1}}) / (- \langle g_k, \eta_k\rangle_{x_k})$ defined in~\eqref{eq:RLS}.
From \cref{thm:CD}, we can similarly propose and analyze the following algorithm, which is a generalization of $\beta_{k+1} = \max\{0,\min\{\beta^{\HS}_{k+1},\beta^{\CD}_{k+1}\}\}$ for the Euclidean case~\cite{andrei2008conjugate}.
\begin{theorem}
\label{thm:LS}
Let $\{x_k\}$ be a sequence generated by \cref{alg:RCGgeneral} under the assumption stated in \cref{thm:CD} and $\beta_{k+1} = \beta_{k+1}^{\text{\rm{\RLSCD}}} :=  \max\{0, \min\{\beta^{\text{\rm{\RLS}}}_{k+1}, \beta^{\text{\rm{\RCD}}}_{k+1}\}\}$.
Then, we have $\liminf_{k \to \infty}\|\grad f(x_k)\|_{x_k} = 0$.
\end{theorem}

\subsubsection{Discussion on R-CG methods with Property (R-$*$)}
In the Euclidean CG methods, Property ($*$) is proposed in~\cite{gilbert1992global}, which is useful for analyzing $\beta_{k+1}$ with the numerator $g_{k+1}^T(g_{k+1}-g_k)$.
We can generalize Property ($*$) to the Riemannian case and refer to the generalized one as Property (R-$*$) as follows:
\begin{definition}[Property (R-$*$)]
We consider \cref{alg:RCGgeneral} and assume that there exist $\varepsilon, \Gamma > 0$ such that $0 < \varepsilon \le \|\grad f(x_k)\|_{x_k} \leq \Gamma$ holds for all $k \ge 0$.
Under this assumption, we say that \cref{alg:RCGgeneral} has \emph{Property (R-$*$)} if there exist constants $b > 1$ and $\lambda > 0$ such that, for all $k \ge 0$, it holds that $|\beta_{k+1}| \le b$ and
\begin{equation*}
    t_k\|\eta_k\| \le \lambda \implies |\beta_{k+1}| \le \frac{1}{2b}. 
\end{equation*}
\end{definition}

We assume that $l_k$ satisfies $l_k \le \|g_k\|_{x_k} / \|\S^{(k)}(g_k)\|_{x_{k+1}}$ in~\eqref{eq:RPRP}--\eqref{eq:RLS} and that there exists a constant $L > 0$ such that the Lipshitz-like condition for the gradient, $\|g_{k+1} - l_k\S^{(k)}(g_k)\|_{x_{k+1}} \le L t_k\eta_k$, holds.
Then, the R-CG methods with $\beta^{\RPRP}_{k+1}$, $\beta^{\RHS}_{k+1}$, and $\beta^{\RLS}_{k+1}$ have Property (R-$*$) under some additional assumptions.

For example, we can prove that the PRP-type of R-CG methods have Property~\mbox{(R-$*$)} as follows.
If $\varepsilon \le \|g_k\|_{x_k} \le \Gamma$, we have
\begin{equation*}
    |\beta^{\RPRP}_{k+1}| \leq \frac{\|g_{k+1}\|_{x_{k+1}}^2 + \|g_{k+1}\|_{x_{k+1}}l_k\|\S^{(k)}(g_k)\|_{x_{k+1}}}{\|g_k\|_{x_k}^2}
    \le \frac{2\Gamma^2}{\varepsilon^2} =: b^{\RPRP}.
\end{equation*}
Furthermore, if $t_k\eta_k \le \lambda^{\RPRP} := \varepsilon^2 / (2L\Gamma b^{\RPRP})$ holds, then
\begin{equation*}
    |\beta^{\RPRP}_{k+1}| \le \frac{\|g_{k+1}\|_{x_{k+1}}\|g_{k+1}-l_k\S^{(k)}(g_k)\|_{x_{k+1}}}{\|g_k\|_{x_k}^2}
    \le \frac{\Gamma L\lambda^{\RPRP}}{\varepsilon^2} = \frac{1}{2b^{\RPRP}}.
\end{equation*}

In the Euclidean space, Property ($*$) can be used to show that the CG methods with $\beta^{\PRP+}_{k+1} := \max\{0, \beta^{\PRP}_{k+1}\}$, $\beta^{\HS+}_{k+1} := \max\{0, \beta^{\HS}_{k+1}\}$, and $\beta^{\LS+}_{k+1} := \max\{0, \beta^{\LS}_{k+1}\}$ have global convergence properties~\cite{gilbert1992global}.
However, it may not be clear how to extend this discussion to the Riemannian case.
This issue is left for future studies.

\subsection{Summary of the theoretical results}
We summarize the discussion in this section as \cref{tab}.
Although our framework addresses a general map $\T^{(k)}$, the table shows only examples of $\T^{(k)}$ for which corresponding previous studies exist.
In the table, each column corresponds to the choice of $\T^{(k)}$ in \cref{alg:RCGgeneral}: the parallel translation $\P$ along the geodesic; differentiated retraction $\mathcal{T}^R$; inverse retraction $(R^{\bw})^{-1}$.
Each row corresponds to the choice of $\beta_{k+1}$ in \cref{alg:RCGgeneral}.
\setlength{\tabcolsep}{1.5mm}
\begin{table}[htbp]
{\footnotesize
  \caption{Summary of the results of previous studies and those of this paper (references shown in parentheses mean that the studies are about specific problems or without convergence analyses)}  \label{tab}
\begin{center}
  \begin{tabular}{|c|c|c|c|} \hline
   & \begin{tabular}{c}$\P$\\ (\cref{ex:T})\end{tabular} & \begin{tabular}{c}$\mathcal{T}^R$\\ (\cref{ex:T})\end{tabular} & 
   \begin{tabular}{c}$(R^{\bw})^{-1}$\\ (\cref{ex:T})\end{tabular}\\ \hline
    \begin{tabular}{c}FR\\ (\cref{thm:FR})\end{tabular}  & (\hspace{-.01mm}\cite{smith1994optimization}) & \cite{ring2012optimization,sato2015new} & \cite{zhu2020riemannian}\\
    \begin{tabular}{c}DY\\ (\cref{thm:DY})\end{tabular}  & (\hspace{-.01mm}\cite{kleinsteuber2007intrinsic}) & \cite{sato2016dai} & \cite{zhu2020riemannian}\\
    \begin{tabular}{c}CD\\ (\cref{thm:CD})\end{tabular}   & -- & -- & --\\
    \begin{tabular}{c}PRP\\ (\cref{subsubsec:PRP})\end{tabular} & (\hspace{-.01mm}\cite{edelman1998geometry,smith1994optimization})   & (\hspace{-.01mm}\cite{AbsMahSep2008})  & --\\
    \begin{tabular}{c}HS\\ (\cref{subsubsec:HS})\end{tabular}   & (\hspace{-.01mm}\cite{kleinsteuber2007intrinsic}) & (\hspace{-.01mm}\cite{boumal2014manopt}) & --\\
    \begin{tabular}{c}LS\\ (\cref{subsubsec:LS})\end{tabular} & (\hspace{-.01mm}\cite{smith1994optimization}) &  --  & --\\
    \begin{tabular}{c}PRP--FR\\ (\cref{thm:PRP})\end{tabular} & --  & -- & --\\
    \begin{tabular}{c}HS--DY\\ (\cref{thm:HS})\end{tabular} & -- & \cite{sakai2020hybrid}   & --\\
    \begin{tabular}{c}LS--CD\\ (\cref{thm:LS})\end{tabular} & -- &  -- & --\\ \hline
    
  \end{tabular}
\end{center}
}
\end{table}

\section{Numerical experiments}
\label{sec:6}
In this section, we compare \cref{alg:RCGgeneral} with several choices of $\beta_{k+1}$.
Specifically, we use $\beta^{\RFR}_{k+1}$, $\beta^{\RDY}_{k+1}$, $\beta^{\RCD}_{k+1}$, $\beta^{\RPRP}_{k+1}$, $\beta^{\RHS}_{k+1}$, and $\beta^{\RLS}_{k+1}$ in~\eqref{eq:RFR}--\eqref{eq:RLS}, and $\beta^{\RPRPFR}_{k+1}$, $\beta^{\RHSDY}_{k+1}$, and $\beta^{\RLSCD}_{k+1}$ in \cref{subsec:PRP_HS_LS} as hybrid methods.
Since one of our contributions is that the proposed class of R-CG methods (\cref{alg:RCGgeneral}) offers the use of a user-selected map $\T^{(k)}$, we deal with two optimization problems using different choices of $\T^{(k)}$.

The experiments were carried out in double-precision floating-point arithmetic on a PC (Intel Xeon CPU E5-2620 v4, 128 GB RAM) equipped with MATLAB R2021b.
In all the experiments below, we implemented the R-CG methods based on \texttt{conjugategradient} in Manopt~\cite{boumal2014manopt}, which is a MATLAB toolbox for Riemannian optimization.
The step length computed in each iteration satisfies the Armijo condition~\eqref{eq:M_armijo} by default.
The iterations of the R-CG methods were terminated when $\|\grad f(x_k)\|_{x_k} / \|\grad f(x_0)\|_{x_0}< 10^{-6}$ was attained.

\subsection{R-CG methods on the product of Grassmann manifolds with the projection-based vector transport for singular value decomposition}

We consider \cref{prob:gen} of large size with $\M := \Grass(p,m) \times \Grass(p,n)$,
where $m = 50,000$, $n = 3,000$, and $p = 100$, implying that $\dim \M = p(m-p) + p(n-p) = 5,280,000$.
Each point $W$ on the Grassmann manifold $\Grass(p,n) \simeq \St(p,n) / \mathcal{O}(p)$ is expressed as an equivalence class $[U] := \{UQ \mid Q \in \mathcal{O}(p)\}$ with a representative $U \in \St(p,n)$.
We define the objective function $f$ on $\M$ as $f([U], [V]) := -\|U^T AV\|_F^2/2$, where $A$ is a randomly generated $m \times n$ matrix and $\|\cdot \|_F$ is the Frobenius norm.

In the experiments, we used the polar-based retraction and the projection-based vector transport $\mathcal{T}^P$, which are the default settings in Manopt's \texttt{grassmannfactory}, and set $\mathscr{T}^{(k)} := \mathcal{T}^P_{t_k\eta_k}( \eta_k)$ and $s_k :=\min\{1, \|\eta_k\|_{x_k}/ \|\T^{(k)}(\eta_k)\|_{x_{k+1}}\}$ in \cref{alg:RCGgeneral}.
Note that this $\T^{(k)}$ satisfies \cref{assump:Tk} as discussed in \cref{ex:P_Grass}.
Similarly, we set $\S^{(k)}(g_k) := \mathcal{T}^P_{t_k\eta_k}(g_k)$ and $l_k := \min\{1, \|g_k\|_{x_k} / \|\S^{(k)}(g_k)\|_{x_{k+1}}\}$ in~\eqref{eq:RPRP}--\eqref{eq:RLS}.
We compared $\beta^{\RSD}_{k+1} := 0$ (steepest descent method) and the nine types of $\beta_{k+1}$ mentioned above (six standard types and three hybrid ones) with the same initial point, which was randomly generated.
\begin{figure}[htbp]
  \centering
  \includegraphics[width = \textwidth]{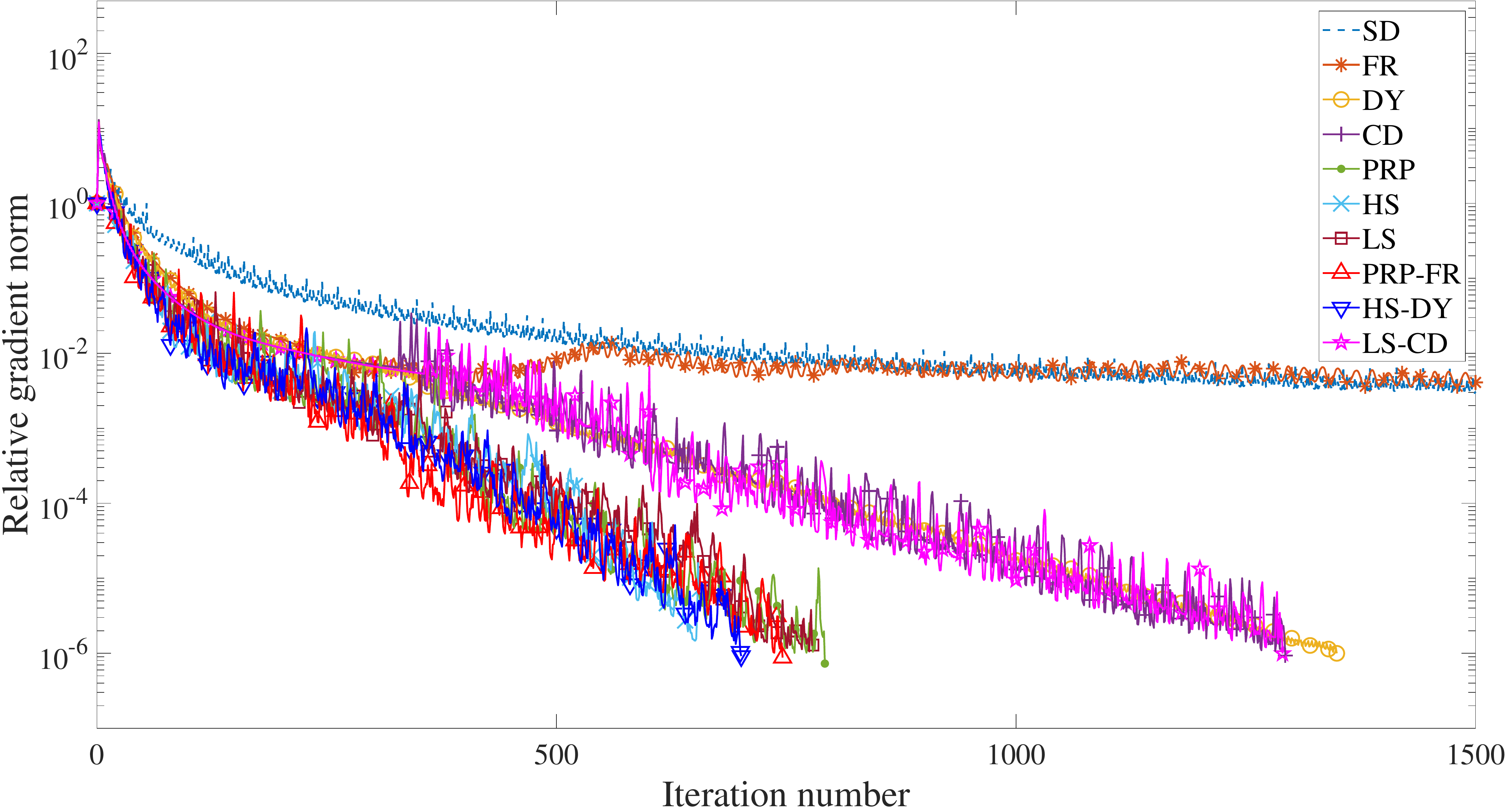}
  \caption{Numerical results for \cref{prob:gen} on $\M = \Grass(p,m) \times \Grass(p,n)$. The horizontal axis represents the iteration number $k$, and the vertical axis represents the relative norm of the gradient of the objective function $\|\grad f(x_k)\|_{x_k} / \|\grad f(x_0)\|_{x_0}$. Markers are put on the graphs at every 20 iterations and at the last iteration for visibility.}
  \label{fig:grassmann}
\end{figure}

\cref{fig:grassmann} shows the convergence histories of the ten methods.
In this figure, we observe several clusters of the graphs: SD and FR; DY, CD, and LS--CD; PRP, HS, LS, PRP--FR, and HS--DY.
We observe that SD is the slowest as expected, and FR, DY, and CD are not so fast either.
These three types of R-CG methods are considered similar as discussed in \cref{sec:5}.
The other three types; PRP, HS, and LS, are faster than FR, DY, and CD.
Regarding the hybrid methods, PRP--FR and HS--DY methods showed similar performances to PRP and HS.
On the other hand, LS--CD is slower than LS.
In this case, LS--CD seems to be slowed down by CD.

We further applied the Riemannian trust-region (TR) method~\cite{absil2007trust} for the same problem with the same initial point and compared computational time.
The time (in seconds) taken for the relative gradient norm to become less than $10^{-6}$ is summarized in \cref{tab:time_Grass}.
As \cref{fig:grassmann} implies, SD and FR did not achieve the stopping criterion within $30$ minutes.
Furthermore, although the convergence of HS was fast, it failed to find a step length satisfying the Armijo condition at $k = 656$.
The minimum value of the relative gradient norm that HS attained was $1.47 \times 10^{-6}$ (at $k = 651$ in $827.7$ seconds).
It is worth noting that the trust-region method took much longer time than most CG methods. Of course, the trust-region method has the advantage of superlinear convergence once a point sufficiently close to an optimal solution is obtained.
However, \cref{tab:time_Grass} shows that the CG methods are competitive enough for the purpose of obtaining a reasonably good solution.

\setlength{\tabcolsep}{1.4mm}
\begin{table}[htbp]
{\footnotesize
  \caption{Computation time [s] required for $\|\grad f(x_k)\|_{x_k} / \|\grad f(x_0)\|_{x_0} < 10^{-6}$ to be satisfied}  \label{tab:time_Grass}
\begin{center}
  \begin{tabular}{|c|c|c|c|c|c|c|c|c|c|c|c|} \hline
   & SD & FR & DY & CD & PRP & HS & LS & PRP--FR & HS--DY & LS--CD & TR\\ \hline
Time [s] & -- & -- & 1594.9 & 1590.6 & 1026.5 & -- & 974.6 & 951.3 & 890.2 & 1610.0 &25961.6 \\
  \hline
  \end{tabular}
\end{center}
}
\end{table}

\subsection{R-CG methods on the manifold of symmetric positive definite matrices for solving Lyapunov equation}

Subsequently, we consider \cref{prob:gen} with $\M := \SPD(n)$, where $n = 50$ and $\M$ is endowed with the Bures--Wasserstein geometry~\cite{malago2018wasserstein}.
We define the objective function $f$ on $\M$ as $f(X) := \tr(XAX) - \tr(XC)$, where $A, C \in \Sym(n)$ are generated in the same way as in (\texttt{Ex2}) of \cite{han2021riemannian}.
The resultant optimization problem on $\M$ is for solving the Lyapunov equation $AX + XA = C$ for $X \in \M$.

We used the exponential retraction and set $\T^{(k)}(\eta_k) := \eta_k$ and $s_k := 1$ in \cref{alg:RCGgeneral}, and $\S^{(k)}(g_k) := g_k$ and $l_k := 1$ for~\eqref{eq:RPRP}--\eqref{eq:RLS}.
We again compared the R-SD and nine types of R-CG methods, where the initial point was $X_0 = I_n$.
\begin{figure}[htbp]
  \centering
  \includegraphics[width = \textwidth]{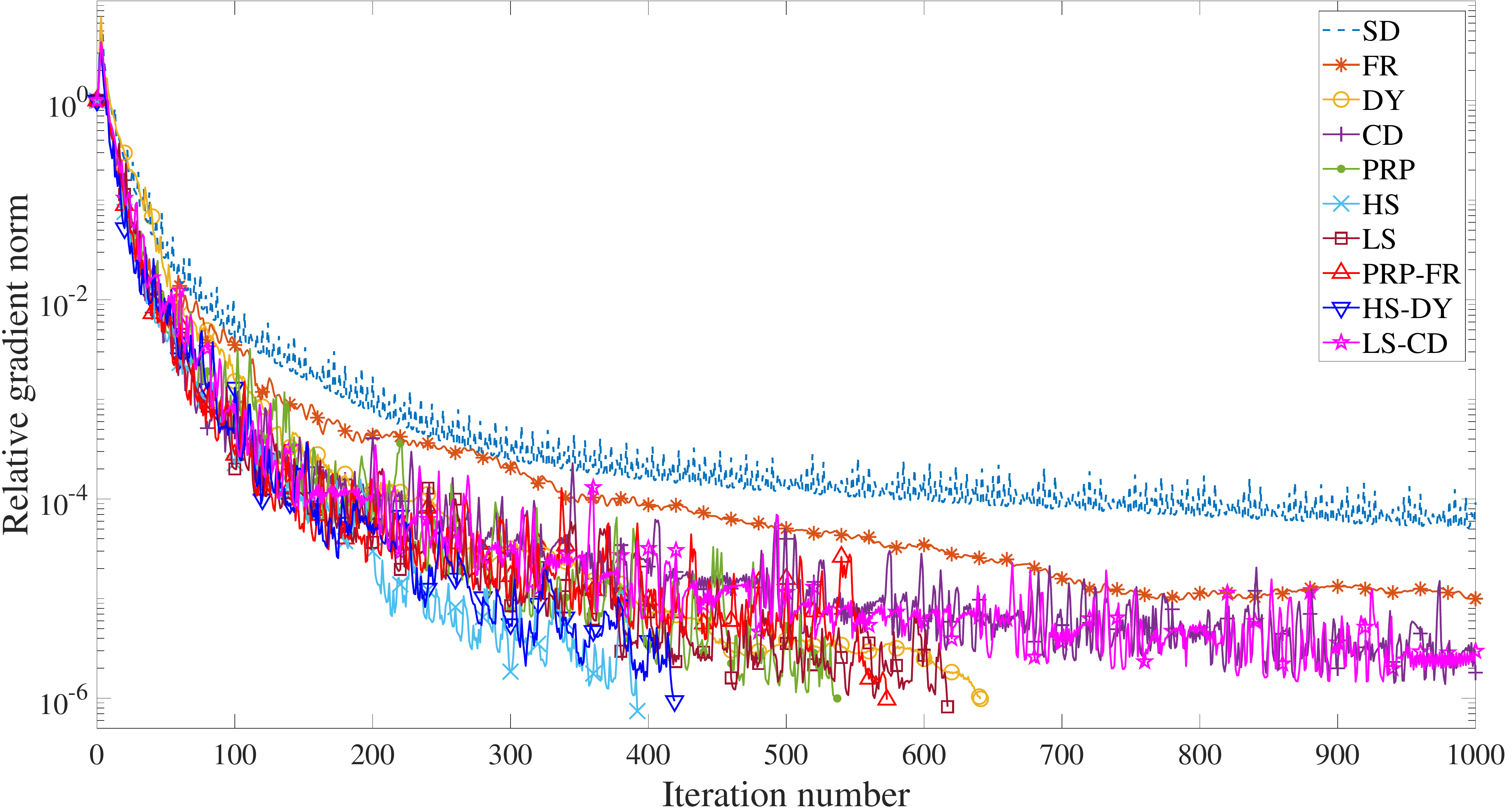}
  \caption{Numerical results for \cref{prob:gen} on $\M = \SPD(n)$. The horizontal axis represents the iteration number $k$, and the vertical axis represents the relative norm of the gradient of the objective function $\|\grad f(X_k)\|_{X_k} / \|\grad f(X_0)\|_{X_0}$. Markers are put on the graphs at every 20 iterations and at the last iteration for visibility.}
  \label{fig:SPD}
\end{figure}

Overall, the results of \cref{fig:SPD} can be explained similarly to those in the previous subsection.
It is observed that SD is the slowest as expected.
Among the R-CG methods, CD is faster than FR and DY, but slower than PRP, HS, LS, and the hybrid methods (except for LS--CD).
LS--CD is slower than PRP--FR and HS--DY possibly because CD negatively affected the performance of LS.

\section{Concluding remarks}
\label{sec:7}
In this paper, to address Riemannian unconstrained optimization problems (\cref{prob:gen}), we proposed a general framework of R-CG methods (\cref{alg:RCGgeneral}) with maps $\T^{(k)}$ and scaling parameters $s_k$ for $k \ge 0$.
Several conditions on $\T^{(k)}$, $s_k$, and the step lengths $t_k$ were developed to provide convergence analyses for the types of \cref{alg:RCGgeneral}.

An important parameter characterizing the (R-)CG methods is $\beta_{k+1}$ in~\eqref{eq:E_direction} (Euclidean case) and~\eqref{eq:RCG_direction} (Riemannian case).
As the six standard types of R-CG methods, we generalized (omitting the subscript $k+1$) $\beta^{\FR}$, $\beta^{\DY}$, $\beta^{\CD}$, $\beta^{\PRP}$, $\beta^{\HS}$, and $\beta^{\LS}$ in Euclidean spaces to the Riemannian counterparts $\beta^{\RFR}$, $\beta^{\RDY}$, $\beta^{\RCD}$, $\beta^{\RPRP}$, $\beta^{\RHS}$, and $\beta^{\RLS}$, respectively.
We extended Zoutendijk's theorem to our proposed framework of the R-CG methods and 
extensively analyzed the FR-, \mbox{DY-,} and CD-types of R-CG methods to guarantee the global convergence properties.
The analyses also claim that any choice of nonnegative $\beta$ that is not smaller than $\beta^{\RFR}$, $\beta^{\RDY}$, or $\beta^{\RCD}$ with appropriate assumptions ensures the global convergence of R-CG methods.
For the PRP-, HS-, and LS-types of R-CG methods, even whose Euclidean versions are not necessarily globally convergent, we discussed why they can outperform the other three types of R-CG methods by explaining that they are considered to be equipped with a restart mechanism when jamming occurs.
Furthermore, modifying them and exploiting the analyses of the FR-, DY-, and CD types of R-CG methods, we proposed several practically and theoretically appropriate methods.

We demonstrated numerical experiments to observe the performances of several R-CG methods in the framework of \cref{alg:RCGgeneral}.
The CD-type of R-CG methods have been rarely used in the literature; however, they are possibly superior to the FR- and DY-type methods.
On the other hand, considering our experiments, the PRP-, HS-, and LS-types of R-CG methods behaved much better than the FR-, DY-, and CD-types.
If guaranteeing the global convergence is important, it is also nice to use the hybrid types, i.e., PRP--FR, HS--DY, and LS--CD types of methods.

Since there are various types of CG methods even for the Euclidean case, this paper does not cover all the existing methods.
However, we believe that the proposed general framework of \cref{alg:RCGgeneral} will be the foundation for future studies on R-CG methods.

\section*{Conflict of interest}
The author declares that he has no conflict of interest.

\bibliographystyle{abbrv}
\bibliography{sato_bib}   

\begin{thebibliography}{10}

\bibitem{absil2007trust}
P.-A. Absil, C.~G. Baker, and K.~A. Gallivan.
\newblock Trust-region methods on {R}iemannian manifolds.
\newblock {\em Foundations of Computational Mathematics}, 7(3):303--330, 2007.

\bibitem{absil2019collection}
P.-A. Absil and S.~Hosseini.
\newblock A collection of nonsmooth {R}iemannian optimization problems.
\newblock In {\em Nonsmooth Optimization and Its Applications}, pages 1--15.
  Springer, 2019.

\bibitem{AbsMahSep2008}
P.-A. Absil, R.~Mahony, and R.~Sepulchre.
\newblock {\em Optimization Algorithms on Matrix Manifolds}.
\newblock Princeton University Press, Princeton, NJ, 2008.

\bibitem{adler2002newton}
R.~L. Adler, J.-P. Dedieu, J.~Y. Margulies, M.~Martens, and M.~Shub.
\newblock {N}ewton's method on {R}iemannian manifolds and a geometric model for
  the human spine.
\newblock {\em IMA Journal of Numerical Analysis}, 22(3):359--390, 2002.

\bibitem{andrei2008conjugate}
N.~Andrei.
\newblock 40 conjugate gradient algorithms for unconstrained optimization. a
  survey on their definition.
\newblock Technical report, ICI Technical report, 2008.

\bibitem{andrei2020nonlinear}
N.~Andrei.
\newblock {\em Nonlinear Conjugate Gradient Methods for Unconstrained
  Optimization}.
\newblock Springer, 2020.

\bibitem{boumal2020introduction}
N.~Boumal.
\newblock An introduction to optimization on smooth manifolds.
\newblock {\em Available online, Aug}, 2020.

\bibitem{boumal2014manopt}
N.~Boumal, B.~Mishra, P.-A. Absil, and R.~Sepulchre.
\newblock Manopt, a {M}atlab toolbox for optimization on manifolds.
\newblock {\em Journal of Machine Learning Research}, 15(1):1455--1459, 2014.

\bibitem{dai1999nonlinear}
Y.-H. Dai and Y.~Yuan.
\newblock A nonlinear conjugate gradient method with a strong global
  convergence property.
\newblock {\em SIAM Journal on Optimization}, 10(1):177--182, 1999.

\bibitem{dai2001efficient}
Y.-h. Dai and Y.~Yuan.
\newblock An efficient hybrid conjugate gradient method for unconstrained
  optimization.
\newblock {\em Annals of Operations Research}, 103(1):33--47, 2001.

\bibitem{edelman1998geometry}
A.~Edelman, T.~A. Arias, and S.~T. Smith.
\newblock The geometry of algorithms with orthogonality constraints.
\newblock {\em SIAM Journal on Matrix Analysis and Applications},
  20(2):303--353, 1998.

\bibitem{edelman1996conjugate}
A.~Edelman and S.~T. Smith.
\newblock On conjugate gradient-like methods for eigen-like problems.
\newblock {\em BIT Numerical Mathematics}, 36(3):494--508, 1996.

\bibitem{fletcher2013practical}
R.~Fletcher.
\newblock {\em Practical Methods of Optimization}.
\newblock John Wiley \& Sons, 2013.

\bibitem{fletcher1964function}
R.~Fletcher and C.~M. Reeves.
\newblock Function minimization by conjugate gradients.
\newblock {\em The Computer Journal}, 7(2):149--154, 1964.

\bibitem{gilbert1992global}
J.~C. Gilbert and J.~Nocedal.
\newblock Global convergence properties of conjugate gradient methods for
  optimization.
\newblock {\em SIAM Journal on Optimization}, 2(1):21--42, 1992.

\bibitem{goto2021approximated}
J.~Goto and H.~Sato.
\newblock Approximated logarithmic maps on {R}iemannian manifolds and their
  applications.
\newblock {\em JSIAM Letters}, 13:17--20, 2021.

\bibitem{hager2005new}
W.~W. Hager and H.~Zhang.
\newblock A new conjugate gradient method with guaranteed descent and an
  efficient line search.
\newblock {\em SIAM Journal on Optimization}, 16(1):170--192, 2005.

\bibitem{hager2006survey}
W.~W. Hager and H.~Zhang.
\newblock A survey of nonlinear conjugate gradient methods.
\newblock {\em Pacific Journal of Optimization}, 2(1):35--58, 2006.

\bibitem{han2021riemannian}
A.~Han, B.~Mishra, P.~Jawanpuria, and J.~Gao.
\newblock On {R}iemannian optimization over positive definite matrices with the
  {B}ures-{W}asserstein geometry.
\newblock {\em arXiv preprint arXiv:2106.00286}, 2021.

\bibitem{hestenes1952methods}
M.~R. Hestenes and E.~Stiefel.
\newblock Methods of conjugate gradients for solving linear systems.
\newblock {\em Journal of Research of the National Bureau of Standards},
  49(6):409--436, 1952.

\bibitem{hu1991global}
Y.~Hu and C.~Storey.
\newblock Global convergence result for conjugate gradient methods.
\newblock {\em Journal of Optimization Theory and Applications},
  71(2):399--405, 1991.

\bibitem{huang2013optimization}
W.~Huang.
\newblock {\em Optimization algorithms on {R}iemannian manifolds with
  applications}.
\newblock PhD thesis, The Florida State University, 2013.

\bibitem{kleinsteuber2007intrinsic}
M.~Kleinsteuber and K.~Huper.
\newblock An intrinsic {CG} algorithm for computing dominant subspaces.
\newblock In {\em 2007 IEEE International Conference on Acoustics, Speech and
  Signal Processing-ICASSP'07}, volume~4, pages IV--1405. IEEE, 2007.

\bibitem{lichnewsky1979une}
A.~Lichnewsky.
\newblock Une methode de gradient conjugue sur des varietes application a
  certains problemes de valeurs propres non lineaires.
\newblock {\em Numerical Functional Analysis and Optimization}, 1(5):515--560,
  1979.

\bibitem{liu1991efficient}
Y.~Liu and C.~Storey.
\newblock Efficient generalized conjugate gradient algorithms, part 1: theory.
\newblock {\em Journal of Optimization Theory and Applications},
  69(1):129--137, 1991.

\bibitem{malago2018wasserstein}
L.~Malag{\`o}, L.~Montrucchio, and G.~Pistone.
\newblock Wasserstein {R}iemannian geometry of {G}aussian densities.
\newblock {\em Information Geometry}, 1(2):137--179, 2018.

\bibitem{narushima2014survey}
Y.~Narushima and H.~Yabe.
\newblock A survey of sufficient descent conjugate gradient methods for
  unconstrained optimization.
\newblock {\em SUT Journal of Mathematics}, 50(2):167--203, 2014.

\bibitem{nocedal2006numerical}
J.~Nocedal and S.~Wright.
\newblock {\em Numerical Optimization, 2nd edn.}
\newblock Springer, 2006.

\bibitem{polak1969note}
E.~Polak and G.~Ribi\'{e}re.
\newblock Note sur la convergence de m{\'e}thodes de directions conjugu{\'e}es.
\newblock {\em ESAIM: Mathematical Modelling and Numerical
  Analysis-Mod{\'e}lisation Math{\'e}matique et Analyse Num{\'e}rique},
  3(R1):35--43, 1969.

\bibitem{polyak1969conjugate}
B.~T. Polyak.
\newblock The conjugate gradient method in extremal problems.
\newblock {\em USSR Computational Mathematics and Mathematical Physics},
  9(4):94--112, 1969.

\bibitem{ring2012optimization}
W.~Ring and B.~Wirth.
\newblock Optimization methods on {R}iemannian manifolds and their application
  to shape space.
\newblock {\em SIAM Journal on Optimization}, 22(2):596--627, 2012.

\bibitem{sakai2020hybrid}
H.~Sakai and H.~Iiduka.
\newblock Hybrid {R}iemannian conjugate gradient methods with global
  convergence properties.
\newblock {\em Computational Optimization and Applications}, 77(3):811--830,
  2020.

\bibitem{sakai2021sufficient}
H.~Sakai and H.~Iiduka.
\newblock Sufficient descent {R}iemannian conjugate gradient methods.
\newblock {\em Journal of Optimization Theory and Applications}, pages 1--21,
  2021.

\bibitem{sato2016dai}
H.~Sato.
\newblock A {D}ai--{Y}uan-type {R}iemannian conjugate gradient method with the
  weak {W}olfe conditions.
\newblock {\em Computational Optimization and Applications}, 64(1):101--118,
  2016.

\bibitem{sato2021riemannian}
H.~Sato.
\newblock {\em Riemannian Optimization and Its Applications}.
\newblock Springer Nature, 2021.

\bibitem{sato2015new}
H.~Sato and T.~Iwai.
\newblock A new, globally convergent {R}iemannian conjugate gradient method.
\newblock {\em Optimization}, 64(4):1011--1031, 2015.

\bibitem{smith1994optimization}
S.~T. Smith.
\newblock Optimization techniques on {R}iemannian manifolds.
\newblock In {\em Hamiltonian and Gradient Flows, Algorithms and Control},
  pages 113--135. American Mathematical Soc., 1994.

\bibitem{zhu2017riemannian}
X.~Zhu.
\newblock A {R}iemannian conjugate gradient method for optimization on the
  {S}tiefel manifold.
\newblock {\em Computational Optimization and Applications}, 67(1):73--110,
  2017.

\bibitem{zhu2020riemannian}
X.~Zhu and H.~Sato.
\newblock Riemannian conjugate gradient methods with inverse retraction.
\newblock {\em Computational Optimization and Applications}, 77(3):779--810,
  2020.

\bibitem{zhu2021cayley}
X.~Zhu and H.~Sato.
\newblock Cayley-transform-based gradient and conjugate gradient algorithms on
  {G}rassmann manifolds.
\newblock {\em Advances in Computational Mathematics}, 47(4):1--28, 2021.

\end{thebibliography}

\end{document}